\documentclass[10pt]{article}
\usepackage{amsthm}
\usepackage{fullpage}
\usepackage{amssymb, amsmath}
\usepackage{hyperref}
\usepackage{graphicx}
\usepackage{titlesec}
\usepackage{enumitem}
\usepackage{comment}
\usepackage{relsize}
\usepackage[font=small]{caption}
\newtheorem{thm}{Theorem}

\newtheorem{lem}[thm]{Lemma}
\newtheorem{cor}[thm]{Corollary}

\newtheorem{prob}[thm]{Problem}
\newtheorem{claim}{Claim}

\title{Constructing graphs with no independent transversals}
\author{Penny Haxell\thanks{Department of Combinatorics and Optimization, University of Waterloo, Waterloo ON Canada. Email: pehaxell@uwaterloo.ca. Partially supported by NSERC}, Ronen Wdowinski\thanks{Department of Combinatorics and Optimization, University of Waterloo, Waterloo ON Canada. Email: ronen.wdowinski@uwaterloo.ca}}
\date{\today}

\begin{document}

\maketitle

\begin{abstract}
Given a graph $G$ and a partition $\mathcal{P}$ of its vertex set, an \textit{independent transversal} (IT) is an independent set of $G$ that contains one vertex from each block in $\mathcal{P}$. Various sufficient conditions for the existence of an IT have been established, and a common theme for many of them is that the block sizes are sufficiently large compared to the maximum degree of $G$. Consequently, there has been interest in constructing graphs with no IT which demonstrate that these bounds on the block sizes are best possible. We describe a simple systematic method for constructing vertex-partitioned graphs with large block sizes and no IT. Unifying previous constructions, we use our method to derive classical extremal constructions due to Jin (1992), Yuster (1997), and Szab\'o and Tardos (2006) in streamlined fashion. For our new results, we describe extremal constructions of minimal graphs with maximum degree two and no IT, generalizing a result of Aharoni, Holzman, Howard, and Spr\"ussel (2015). We construct a family of locally sparse graphs with no IT, complementing an asymptotic result of Loh and Sudakov (2007). We describe new and smaller counterexamples to a list coloring conjecture of Reed (1999), which was originally disproved by Bohman and Holzman (2002). We disprove a conjecture of Aharoni, Alon, and Berger (2016) about IT's in graphs without large induced stars. We answer negatively a question of Aharoni, Holzman, Howard, and Spr\"ussel (2015) about extremal graphs with no IT, but we also prove that a useful variation of their question does hold.
\end{abstract}

\section{Introduction}
Let $G$ be a graph and let $\mathcal{P} = \{V_1, \ldots, V_r\}$ be a partition of its vertex set $V(G)$ into \textit{blocks} (partition classes) $V_i$. An \textit{independent transversal} (IT) of $G$ with respect to $\mathcal{P}$ is an independent set $\{v_1, \ldots, v_r\}$ of $G$ such that $v_i \in V_i$ for every $i \in [r]$. Independent transversals have many combinatorial applications, such as to SAT \cite{Ha3}, linear arboricity \cite{Al1, Al2}, hypergraph matchings \cite{AhHa}, list coloring \cite{Ha2}, circular coloring~\cite{Kaetal}, linear groups~\cite{Bretal}, resource allocation \cite{AsFeSa, HaSz2}, and many others (see e.g.~\cite{GraHa} and the references therein). It is NP-complete to decide whether a given vertex-partitioned graph has an IT, but due to the above applications there has been great interest in sufficient conditions for the existence of an IT. One frequently applied sufficient condition is the following theorem from \cite{Ha1, Ha2}.

\begin{thm}[\cite{Ha1, Ha2}] \label{max-degree-IT}
Let $G$ be a graph with maximum degree $d$, and let $\mathcal{P} = \{V_1, \ldots, V_r\}$ be a partition of its vertex set $V(G)$. If $|V_i| \ge 2d$ for every $i$, then $G$ has an IT with respect to $\mathcal{P}$.
\end{thm}

Theorem \ref{max-degree-IT} was an improvement over Alon's earlier result \cite{Al2} that an IT exists whenever $|V_i| \ge 2ed$ for every $i$. There is a close connection between IT's of graphs and the topological connectedness of their independence complexes. Theorem \ref{max-degree-IT} and variants of it have been proven using these fruitful topological ideas \cite{AhAlBe, AhBe1, AhChKo, AhHa, Ha2, Me1, Me2}, but the original proof of Theorem \ref{max-degree-IT} was combinatorial, and this combinatorial approach has also been used for various other related settings \cite{AhHa, BeHaSz, Ha1, Ha2, Ha3}. For more on the topological point of view, we refer to the survey \cite{Ha4}.

Szab\'o and Tardos \cite{SzTa} proved that the condition $|V_i| \ge 2d$ in Theorem \ref{max-degree-IT} is best possible for every integer $d \ge 1$, in the sense that there exist graphs $G$ with maximum degree $d$ and vertex partitions $\mathcal{P} = \{V_1, \ldots, V_r\}$ of $V(G)$ such that we cannot relax the condition $|V_i| \ge 2d$ to $|V_i| \ge 2d - 1$ for every $i$ while still guaranteeing that $G$ has an IT.

\begin{thm}[\cite{SzTa}] \label{Szabo-Tardos}
Let $d \ge 1$ and let $G$ be the disjoint union of $2d - 1$ copies of the complete bipartite graph $K_{d,d}$. There exists a partition $\mathcal{P} = \{V_1, \ldots, V_{2d}\}$ of $V(G)$ such that $|V_i| = 2d - 1$ for every $i$ and $G$ has no IT.
\end{thm}

This theorem was previously proven by Jin \cite{Ji} and Yuster \cite{Yu} independently when $d$ is a power of $2$, but their partitions $\mathcal{P}$ were different from those of Szab\'o and Tardos \cite{SzTa}. Actually, Szab\'o and Tardos \cite{SzTa} proved a more general version of Theorem \ref{Szabo-Tardos} where the block sizes $|V_i|$ depend on the number of blocks $r$ in addition to the maximum degree $d$ (see Corollary \ref{general-Szabo-Tardos}). Together with a refinement of Theorem \ref{max-degree-IT} proven in \cite{HaSz1}, this more general construction of Szab\'o and Tardos \cite{SzTa} solved a problem of Bollob\'as, Erd\H{o}s, and Szemer\'edi \cite{BoErSz}, who were interested in the complementary, Tur\'an-type problem of minimum degree conditions for the existence of a clique $K_r$ in a balanced $r$-partite graph. In addition to the above mentioned constructions \cite{Ji, SzTa, Yu}, many other non-trivial constructions of graphs with no IT have been described with different applications \cite{AhHoHoSp, Al3, BoHo, GrKaTrWa}. However, there has not been an explicit systematic procedure for constructing such examples of graphs with no IT. 

In this paper, we describe a simple iterative method that can be used to derive essentially all of the constructions cited above. Our construction method centers around the following simple lemma, which was also used in \cite{HaWd, CaHaKaWd}.

\begin{lem} \label{join-lemma}
Let $G$ and $H$ be disjoint graphs, and let $\mathcal{P} = \{V_1, \ldots, V_r\}$ and $\mathcal{Q} = \{W_1, \ldots, W_s\}$ be partitions of $V(G)$ and $V(H)$, respectively. Let $\mathcal{R} = \{V_1', \ldots, V_r', W_1, \ldots, W_{s-1}\}$, where $V_1' \supseteq V_1, \ldots, V_r' \supseteq V_r$ are obtained by distributing each of the vertices of $W_s$ into one of $V_1, \ldots, V_r$ arbitrarily. If $G$ has no IT with respect to $\mathcal{P}$ and $H$ has no IT with respect to $\mathcal{Q}$, then $G \cup H$ has no IT with respect to $\mathcal{R}$.
\end{lem}

\begin{proof}
Assume for contradiction that $G \cup H$ has an IT $\{v_1, \ldots, v_r, w_1, \ldots, w_{s-1}\}$ with respect to $\mathcal{R}$, where $v_i \in V_i'$ for $1 \le i \le r$ and $w_j \in W_j$ for $1 \le j \le s-1$. If $v_i \in V_i$ for every $1 \le i \le r$, then $\{v_1, \ldots, v_r\}$ is an IT of $G$ with respect to $\mathcal{P}$, a contradiction. So suppose instead that $v_j \in V_j' \cap W_s$ for some $1 \le j \le r$. Then $\{w_1, \ldots, w_{s-1}, v_j\}$ is an IT of $H$ with respect to $\mathcal{Q}$, again a contradiction.
\end{proof}

Notice that the complete bipartite graph $K_{d,d}$ has no IT with respect to its standard bipartition (the vertex partition into two independent sets of size $d$). We will show that we can reconstruct the partitions of Jin \cite{Ji}, Yuster \cite{Yu}, and Szab\'o and Tardos \cite{SzTa} for Theorem \ref{Szabo-Tardos} effectively by applying Lemma \ref{join-lemma} iteratively on $K_{d,d}$ in different ways. Later, we will show that in fact every construction for Theorem \ref{Szabo-Tardos} can be derived by applying Lemma \ref{join-lemma} iteratively on $K_{d,d}$ (Corollary \ref{characterize-Kdd}). Besides reproving Theorem \ref{Szabo-Tardos}, we will use Lemma \ref{join-lemma} to find many new constructions of graphs with no IT that answer some open questions. 

Let $G$ be a graph and let $\mathcal{P} = \{V_1, \ldots, V_r\}$ be a partition of $V(G)$. Using topological methods, Aharoni, Holzman, Howard, and Spr\"ussel \cite{AhHoHoSp} proved that if $G$ has maximum degree $d > 2$, $|V_i| \ge 2d - 1$ for every $i$, and $G$ has no IT, then $G$ necessarily contains the disjoint union of $2d - 1$ copies of $K_{d,d}$. Therefore, Theorem \ref{Szabo-Tardos} verifies that the disjoint union of $2d - 1$ copies of $K_{d,d}$ is the unique minimal obstruction to the existence of an IT when the graph has maximum degree $d > 2$ and the block sizes satisfy $|V_i| \ge 2d - 1$. On the other hand, for the case of maximum degree $d = 2$ and $|V_i| \ge 3$, the same authors proved that $G$ necessarily contains the disjoint union of three cycles of lengths $1$ modulo $3$. This result is related to the cycle-plus-triangles problem \cite{AhBeZi, DuHsHw, FlSt, Sa}. One of our new results is verifying that, indeed, all disjoint unions of three cycles of lengths $1$ modulo $3$ are minimal obstructions to the existence of an IT (when $d = 2$ and $|V_i| \ge 3$ for every $i$).

\begin{thm} \label{cycles-no-IT}
Let $G = C_{\ell_1} \sqcup C_{\ell_2} \sqcup C_{\ell_3}$ be the disjoint union of three cycles with lengths $\ell_j \equiv 1 \pmod{3}$, $\ell_j \ge 4$ for $1 \le j \le 3$. There exists a partition $\mathcal{P} = \{V_1, \ldots, V_r\}$ of $V(G)$ such that $|V_i| = 3$ for every $i$ and $G$ has no IT.
\end{thm}

For comparison, Aharoni, Holzman, Howard, and Spr\"ussel  \cite{AhHoHoSp} showed that for any $\ell \equiv 1 \pmod{3}$, $\ell \ge 4$, there exist $k \ge 1$ and a partition $\mathcal{P} = \{V_1, \ldots, V_r\}$ of the vertex set of the graph $G = C_\ell \sqcup \cdots \sqcup C_\ell$ ($k$ cycles of length $\ell$) such that $|V_i| = 3$ for every $i$ and $G$ has no IT. They showed that one can take $k = \frac{\ell}{2} + 1$ if $\ell$ is even, and $k = \ell + 2$ if $\ell$ is odd. Our result says that in fact one can always take $k = 3$, and the cycle lengths do not have to be the same.

Although the block size lower bound $2d$ in Theorem \ref{max-degree-IT} is best possible, there has been interest in decreasing it in some contexts. These decrements are achieved by restricting either the graphs $G$ or the types of vertex partitions $\mathcal{P}$ under consideration, so as to avoid structures present in the known constructions for Theorem \ref{Szabo-Tardos} \cite{Ji, SzTa, Yu}. For example, one approach has been to forbid the graph $G$ from containing large induced complete bipartite subgraphs such as $K_{d,d}$ or even $K_{1,k}$ \cite{AhAlBe, AhBeZi, AhHoHoSp, HaWd}. In this direction, we will disprove a conjecture of Aharoni, Alon, and Berger \cite{AhAlBe} and answer negatively a question of Aharoni, Holzman, Howard, and Spr\"ussel \cite{AhHoHoSp} (see below). Another approach for decreasing the block size lower bound $2d$ has been to restrict the structure of the partition $\mathcal{P}$ around the graph $G$. This latter approach will be our next focus.

First we look at the setting of locally sparse graphs. The \textit{local degree} of $G$ is the maximum, over all $v \in V(G)$ and $V_i \in \mathcal{P}$ not containing $v$, of the number of edges between $v$ and $V_i$. The \textit{multiplicity} of $G$ is the maximum, over all connected components $C$ of $G$ and $V_i \in \mathcal{P}$, of the number of vertices in $C \cap V_i$. Notice that the local degree of $G$ is always at most the multiplicity of $G$. Loh and Sudakov \cite{LoSu} proved that if $G$ has maximum degree $d$, local degree $m = o(d)$, and $|V_i| \ge d + o(d)$ for every $i$, then $G$ has an IT. We call such vertex-partitioned graphs, whose local degree grows much more slowly than the maximum degree, \textit{locally sparse}. Thus the $2d$ lower bound in Theorem \ref{max-degree-IT} decreases significantly when we assume that the vertex-partitioned graph $G$ is locally sparse. More recently, Glock and Sudakov \cite{GlSu} and Kang and Kelly \cite{KaKe} independently strengthened Loh and Sudakov's \cite{LoSu} result to a \textit{maximum block average degree} setting instead of a maximum degree setting. Towards the effort of finding the optimal block sizes for locally sparse graphs that guarantee an IT, we apply Lemma \ref{join-lemma} to prove the following.

\begin{thm} \label{locally-sparse-graph}
For all integers $d \ge m \ge 1$ where $d$ is a multiple of $m$, there exists a graph $G$ and partition $\mathcal{P}$ of $V(G)$ such that $G$ has maximum degree $d$, local degree and multiplicity $m$, $|V_i| \ge d + 2m - \left\lceil \frac{2m^2 + m}{d + m}\right\rceil$ for every $V_i \in \mathcal{P}$, and there is no IT.
\end{thm}

In particular, when $d \ge 2m^2$, we can take the block sizes to satisfy $|V_i| \ge d + 2m - 1$ for every $V_i \in \mathcal{P}$ in Theorem \ref{locally-sparse-graph}. Theorem \ref{locally-sparse-graph} was previously proven when $m = 1$ by Bohman and Holzman \cite{BoHo}, with a slightly simplified construction described in \cite{AhHoHoSp}. On the other hand, even in the case of local degree $m = 1$, Loh and Sudakov \cite{LoSu} only showed that there exists some $\epsilon > 0$ such that for all sufficiently large $d$, if $|V_i| \ge d + d^{1 - \epsilon}$ for every $i$ then there exists an IT. Thus, there is still a large gap between Loh and Sudakov's upper bound and our lower bound on the optimal block sizes that guarantee the existence of an IT in graphs with a given maximum degree $d$ and local degree $m$.

Work on locally sparse graphs was inspired by a list coloring conjecture of Reed \cite{Re} about \emph{maximum color degree} (see Section \ref{section-list-coloring} for definitions). The construction of Bohman and Holzman \cite{BoHo} mentioned above was actually a counterexample to Reed's conjecture. Given a graph $H$ and list assignment $L$ for $V(H)$, the problem of finding a proper $L$-coloring of $H$ can easily be turned into the problem of finding an IT in an auxiliary graph $G$ with vertex partition $\mathcal{P}$ (see Section \ref{section-list-coloring} for details). The vertex-partitioned graphs $G, \mathcal{P}$ arising in this fashion are characterized by the property of having multiplicity one (in particular, they are locally sparse), together with a consistent adjacency property across all connected components (see Lemma \ref{list-cover-graph}). Reed \cite{Re} conjectured that if $H, L$ has maximum color degree $d$ and $|L(x)| \ge d+1$ for every $x \in V(H)$, then $H$ has a proper $L$-coloring. Reed and Sudakov \cite{ReSu} proved that there exists a proper $L$-coloring under the assumption that $|L(x)| \ge d+o(d)$ for every $x \in V(H)$. The asymptotic result of Loh and Sudakov \cite{LoSu} for locally sparse graphs mentioned above is a generalization of this result of Reed and Sudakov. Aharoni, Berger, and Ziv \cite{AhBeZi0} proved Reed's conjecture for chordal graphs. On the other hand, Bohman and Holzman \cite{BoHo} disproved Reed's conjecture for every $d \ge 2$, by constructing an example where $|L(x)| = d+1$ for every $x \in V(H)$ and there is no proper $L$-coloring. They describe their example by exhibiting an appropriate vertex-partitioned graph with no IT (see Theorem \ref{Bohman-Holzman}). We will use Lemma \ref{join-lemma} to produce new and smaller counterexamples to the conjecture of Reed, which involve $4(d+1)^2$ blocks compared to the $2(d+1)(d^2 + 1)$ blocks employed by Bohman and Holzman \cite{BoHo}.

Our next topics of interest are a conjecture of Aharoni, Alon, and Berger \cite{AhAlBe} and a question of Aharoni, Holzman, Howard, and Spr\"ussel \cite{AhHoHoSp}. Aharoni, Alon, and Berger \cite{AhAlBe} were interested in the existence of IT's in $K_{1,k}$-free graphs, meaning graphs that do not contain the star $K_{1,k}$ as an induced subgraph. They conjectured that for every $K_{1,k}$-free graph $G$ and partition $\mathcal{P} = \{V_1, \ldots, V_r\}$ of $V(G)$, if $|V_i| \ge d + k - 1$ for every $i$, then $G$ has an IT. This would improve on the $2d$ bound in Theorem \ref{max-degree-IT}, where the extremal constructions contain many induced complete bipartite graphs and thus many large induced stars. Aharoni, Alon, and Berger \cite{AhAlBe} proved that an IT exists if $|V_i| \ge \frac{d(2k - 3) + k - 1}{k - 1}$ for every $i$, and they proved their conjecture when $G$ is the line graph of a $(k - 1)$-uniform linear hypergraph. We will construct counterexamples to their conjecture for all $k \ge 3$. Next, recall that Aharoni, Holzman, Howard, and Spr\"ussel \cite{AhHoHoSp} proved that if $G$ has maximum degree $d > 2$, $|V_i| \ge 2d - 1$ for every $i$, and there is no IT, then $G$ contains the disjoint union of $2d - 1$ copies of the complete bipartite graph $K_{d,d}$. These authors asked whether, additionally, there always exists a copy of $K_{d,d}$ in $G$ that is contained in the union of two blocks $V_i, V_j$. We give a construction showing that the answer is negative.

Despite the negative answer to the question of Aharoni, Holzman, Howard, and Spr\"ussel \cite{AhHoHoSp}, we prove that a variation of their question does hold, namely when the graph $G$ is exactly the disjoint union of $2d - 1$ copies of $K_{d,d}$, as in Theorem \ref{Szabo-Tardos} (Corollary \ref{union-of-two-classes-Kdd}). In fact, this variation is relevant for proving one of our other results mentioned above, that all possible constructions $G, \mathcal{P}$ for Theorem \ref{Szabo-Tardos} can be derived by iteratively applying Lemma \ref{join-lemma} on $K_{d,d}$ (Corollary \ref{characterize-Kdd}). We discuss these latter two results in the more general settings of Theorem \ref{union-of-two-classes} and Theorem \ref{characterize-complete-bipartite}, in Sections \ref{section-2blocks} and \ref{section-characterize-complete-bipartite} respectively. The proof of Theorem \ref{union-of-two-classes} is slightly technical and deferred to Section \ref{big-proof}.

Our emphasis in this paper is on simple constructions with short and elementary proofs (with the only possible exception being Section \ref{big-proof}). However, our methods can be used to prove a wide variety of further theorems and generalizations of known results in other contexts. We describe some of these in Section~\ref{section-further}. In addition to Theorem \ref{union-of-two-classes} and Theorem \ref{characterize-complete-bipartite} mentioned above, in Section~\ref{section-further} we also note how Lemma \ref{join-lemma} and our construction method easily generalize to all simplicial complexes (Lemma \ref{join-lemma-complex}), thus significantly increasing their scope. We will briefly discuss what this generalized method yields for some simplicial complexes of interest, specifically independent transversals in hypergraphs and $H$-free transversals in graphs for a fixed regular graph $H$. Details of these topics plus more are given in \cite{WdThesis}.

The rest of this paper is organized as follows. In each section except for the last one, we will be applying Lemma \ref{join-lemma} iteratively in some way to construct graphs with no IT. In Section \ref{section-disjoint-complete-bipartite} we will explain how to derive constructions for Theorem \ref{Szabo-Tardos}, and in Section \ref{section-d-equals-2} we will prove Theorem \ref{cycles-no-IT}. Section \ref{section-locally-sparse-graphs-1} is devoted to the proof of Theorem \ref{locally-sparse-graph}, and in Section \ref{section-list-coloring} we will derive new counterexamples to the list coloring conjecture of Reed \cite{Re}. In Section \ref{section-questions}, we will construct counterexamples to the conjecture of Aharoni, Alon, and Berger \cite{AhAlBe} and to the question of Aharoni, Holzman, Howard, and Spr\"ussel \cite{AhHoHoSp}. In Section \ref{section-further}, we will discuss how Theorem \ref{union-of-two-classes} does verify an alternate version of their question, while also implying Theorem \ref{characterize-complete-bipartite}. We will also discuss how our method generalizes to all simplicial complexes. In Section \ref{big-proof}, we will prove Theorem \ref{union-of-two-classes}.

\section{Independent transversals in bounded degree graphs}
In this section, we will demonstrate the usefulness of Lemma \ref{join-lemma} by (re)proving Theorem \ref{Szabo-Tardos} and by proving Theorem \ref{cycles-no-IT}. Together with Theorem \ref{many-copies-Kdd} and Theorem \ref{cycles-length-1-mod-3} below (both proved in~\cite{AhHoHoSp}), these two theorems solve the following extremal problem: Find all the minimal graphs $G$ with maximum degree at most $d$ which have no IT with respect to some vertex partition $\mathcal{P}$ into blocks of size at least $2d - 1$. These graphs $G$ are the disjoint union of $2d - 1$ copies of $K_{d,d}$ when $d \neq 2$, and every disjoint union of three cycles of lengths $1$ modulo $3$ when $d = 2$.

\subsection{Disjoint unions of the complete bipartite graph $K_{d,d}$} \label{section-disjoint-complete-bipartite}
First we show how to prove Theorem \ref{Szabo-Tardos}. We will first show how to derive the vertex partitions of Szab\'o and Tardos \cite{SzTa}, then those of Jin \cite{Ji} and Yuster \cite{Yu}.

Let $G_0 = K_{d,d}$ and let $\mathcal{P}_0 = \{V_1, V_2\}$ be the standard bipartition of $G_0$. Note that $|V_1| = |V_2| = d$ and that $K_{d,d}$ has no IT with respect to $\mathcal{P}$. Our goal is to enlarge both $V_1$ and $V_2$ into blocks of size $2d - 1$ while preserving the property of the graph having no IT. We only show how to enlarge $V_1$, since enlarging $V_2$ can be done symmetrically afterwards. Throughout this procedure, we will have a distinguished ``small" vertex block $U_k$, starting with $U_0 = V_1$, that continually increases in size as we add copies of $K_{d,d}$, while all other vertex blocks (apart from $V_2$) always have size $2d - 1$. Eventually $U_k$ will also have size $2d - 1$, which finishes the required construction. 

Let $k \ge 1$, and let $G_{k-1}$ be the disjoint union of $k$ copies of $K_{d,d}$. Assume we have a partition $\mathcal{P}_{k-1}$ of $V(G_{k-1})$ such that: there exists a block $U_{k-1} \in \mathcal{P}_{k-1}$ of size $|U_{k-1}| = d+k-1$, and $|U| = 2d - 1$ for every $U \in \mathcal{P}_{k-1} \setminus \{V_2, U_{k-1}\}$. Let $G_k = G_{k-1} \sqcup K_{d,d}$, and let $\mathcal{P}' = \{V_1', V_2'\}$ be the standard bipartition of this newly added copy of $K_{d,d}$. We let $\mathcal{P}_k$ be the partition of $V(G_k)$ obtained by first taking the union of the partitions $\mathcal{P}_{k-1}$ and $\mathcal{P}'$, and then distributing the vertices of $U_{k - 1}$ into the blocks $V_1'$ and $V_2'$. Specifically, we distribute $d - 1$ vertices of $U_{k - 1}$ into $V_2'$ to make this modified block have size $2d - 1$, and we distribute the remaining $k$ vertices of $U_{k - 1}$ into $V_1'$ to make this modified block, which we now call $U_k$, have size $d + k$. Then $G_k$ has no IT with respect to $\mathcal{P}_k$ by Lemma \ref{join-lemma}, so the procedure may continue. Eventually when $k = d - 1$, we have that all blocks of $\mathcal{P}_{d-1}$ (apart from $V_2$) have size $2d - 1$, so we may stop. Enlarging $V_2$ in symmetric fashion, by adding another $d - 1$ copies of $K_{d,d}$, leads to a partition $\mathcal{P}$ satisfying Theorem \ref{Szabo-Tardos}.

\begin{figure}
	\centering
	\includegraphics[scale=0.8]{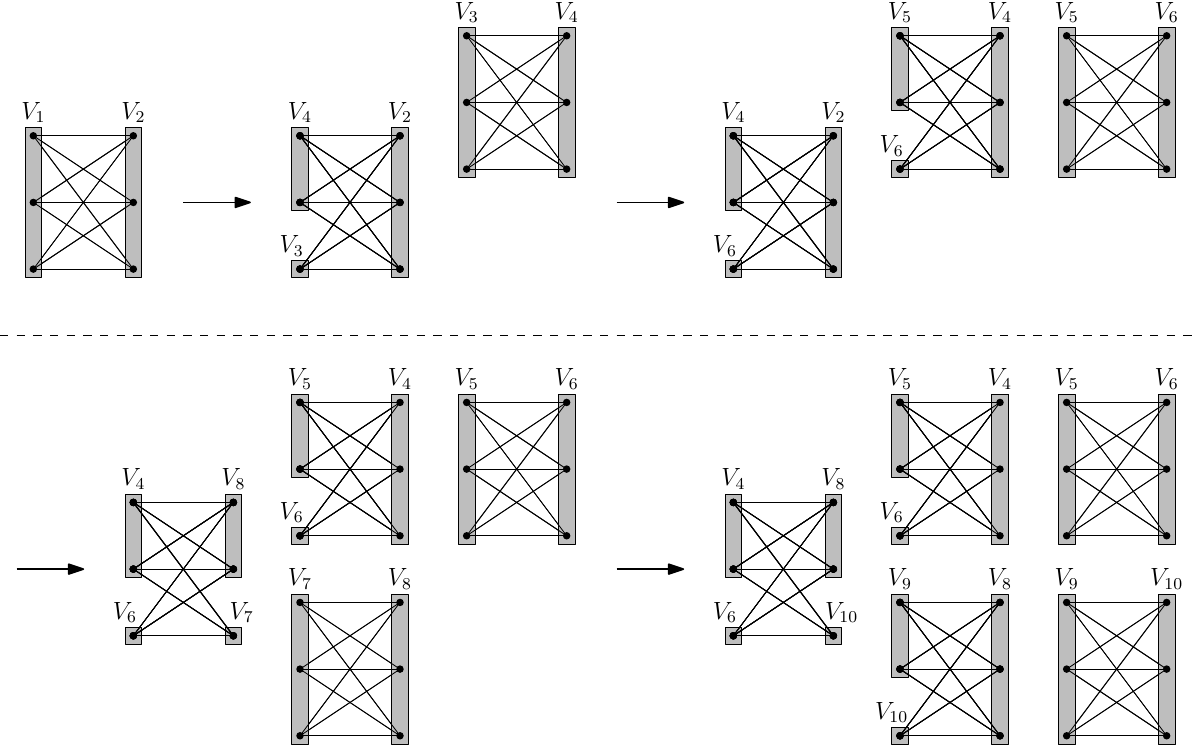}
	\caption{Deriving the Szab\'o-Tardos construction for $d = 3$ by repeatedly applying Lemma \ref{join-lemma}. We start with $K_{3,3}$ and the standard bipartition $\{V_1, V_2\}$. In the first step, we add a new copy of $K_{3,3}$ and its standard bipartition $\{V_3, V_4\}$, and then we distribute one vertex of $V_1$ into $V_3$ and the other two vertices of $V_1$ into $V_4$. In the second step, we add a new copy of $K_{3,3}$ with its standard bipartition $\{V_5, V_6\}$, and then we distribute two vertices of $V_3$ into $V_5$ and the other two vertices of $V_3$ into $V_6$ as shown. The remaining two steps are analogous.}
	\label{complete-bipartite-figure}
\end{figure}

The partition of Szab\'o and Tardos \cite{SzTa} can be obtained exactly from this procedure by choosing, at each step $k$, the $k$ vertices of $U_{k - 1}$ that get distributed into $V_1'$ to include one vertex from each of the $k$ copies of $K_{d,d}$ in $G_{k - 1}$. See Figure \ref{complete-bipartite-figure} for an illustration of this procedure when $d = 3$. Eventually when $k = d - 1$, the distinguished block $U_{d-1}$ forms a special ``leftover" block in the construction Szab\'o and Tardos (which they label as $V_{q+1}$). We get their additional leftover block (which they label as $V_{2q+2}$) after doing the symmetric procedure on $V_2$.

For $d \ge 4$, the partition of Szab\'o and Tardos is not the only possible result of the above procedure. Indeed, Szab\'o and Tardos themselves observed that when $d \ge 4$ is a power of $2$, their partitions $\mathcal{P}$ for Theorem \ref{Szabo-Tardos} differ from those of Jin \cite{Ji} and Yuster \cite{Yu}, despite all of them being on the same graph $G$ (namely the disjoint union of $2d - 1$ copies of $K_{d,d}$). Due to the somewhat arbitrary nature in which we can distribute vertices according to Lemma \ref{join-lemma}, the fact that there exist such different partitions of the same graph $G$ with no IT is not surprising. 

When $d$ is a power of $2$, the partitions of Yuster \cite{Yu} can be derived using the following slight modification of the above procedure. Instead of distributing vertices so as to always have one distinguished ``small" block, at each step we distribute the vertices into $V_1'$ and $V_2'$ in equitable fashion, with half of them going to $V_1'$ and half of them going to $V_2'$. The block of $G_{k-1}$ that we choose to distribute to $V_1'$ and $V_2'$ at each step is any of the current smallest ones. The partitions of Jin \cite{Ji} are special cases of those of Yuster. Specifically, one performs this described procedure in such a way as to eventually make one copy of $K_{d,d}$ ``colorful", containing one vertex from each of the $2d$ blocks. Thus we can also derive Jin's partitions using Lemma \ref{join-lemma}. One could come up with many other schemes that result in partitions $\mathcal{P}$ satisfying Theorem \ref{Szabo-Tardos}. In fact, we will prove all possible constructions for Theorem \ref{Szabo-Tardos} can be derived in this kind of fashion of applying Lemma \ref{join-lemma} iteratively on $K_{d,d}$ with the standard bipartition (Corollary \ref{characterize-Kdd}).

We remark that we do not have to stop at $2d - 1$ copies of $K_{d,d}$. Indeed, we may continually add further copies of $K_{d,d}$, and there will by necessity be some blocks of size greater than $2d - 1$ after this. However, if the number of $K_{d,d}$ components is $k$, then the resulting graph has $2dk$ vertices partitioned into $k+1$ blocks, so in particular there will always be a block of size at most $2d - 1$, in accordance with Theorem \ref{max-degree-IT}.

\subsection{The case d = 2: Disjoint unions of cycles of lengths 1 modulo 3} \label{section-d-equals-2}

Next, we prove Theorem \ref{cycles-no-IT}. Recall that Aharoni, Holzman, Howard, and Spr\"ussel \cite{AhHoHoSp} proved the following Brooks-type theorem.

\begin{thm}[\cite{AhHoHoSp}] \label{many-copies-Kdd}
Let $G$ be a graph with maximum degree $d > 2$, and let $\mathcal{P} = \{V_1, \ldots, V_r\}$ be a partition of $V(G)$. If $|V_i| \ge 2d - 1$ for every $i$ but $G$ has no IT, then $G$ contains the disjoint union of $2d - 1$ copies of $K_{d,d}$.
\end{thm}

Thus, it is not a coincidence that for all previous constructions achieving block sizes $|V_i| \ge 2d - 1$ and no IT \cite{Ji, SzTa, Yu}, the underlying graph $G$ is the disjoint union of $2d - 1$ copies of $K_{d,d}$. This is despite different authors sometimes describing different vertex partitions $\mathcal{P}$. As with Brooks' theorem, the case of maximum degree $d = 2$ is slightly more complicated and interesting than the case $d > 2$. When $G$ is one cycle and $|V_i| = 3$ for every $i$, the existence of an IT was a conjecture of Du, Hsu, and Hwang \cite{DuHsHw}. Erd\H{o}s \cite{Er} popularized a stronger 3-colorability version of the conjecture, and this became known as the ``cycle-plus-triangles problem". The conjecture was proven by Fleischner and Stiebitz \cite{FlSt} and by Sachs \cite{Sa}. For more general graphs with maximum degree $2$, Theorem \ref{Szabo-Tardos} shows that when $G$ contains the disjoint union of three $4$-cycles $C_4$, the condition $|V_i| \ge 3$ does not guarantee the existence of an IT. But in contrast with Theorem \ref{many-copies-Kdd}, three disjoint $4$-cycles is not the only obstruction to the existence of an IT when $d = 2$. Aharoni, Holzman, Howard, and Spr\"ussel  \cite{AhHoHoSp} also proved the following result suggesting that other cycle lengths are possible obstructions to the existence of an IT.

\begin{thm}[\cite{AhHoHoSp}] \label{cycles-length-1-mod-3}
Let $G$ be a graph with maximum degree $d = 2$, and let $\mathcal{P} = \{V_1, \ldots, V_r\}$ be a partition of $V(G)$. If $|V_i| \ge 3$ for every $i$ but $G$ has no IT, then $G$ contains the disjoint union of three cycles of lengths $1 \pmod{3}$.
\end{thm}

We now prove Theorem \ref{cycles-no-IT} which shows that, indeed, all disjoint unions of three cycles of lengths $1 \pmod{3}$ are minimal obstructions to existence of an IT in the context of graphs with maximum degree $d = 2$ and block sizes $|V_i| \ge 3$. To prove the theorem, we employ the following lemma that was also proven and utilized by Aharoni, Holzman, Howard, and Spr\"ussel  \cite{AhHoHoSp} in their constructions.

\begin{lem}[\cite{AhHoHoSp}] \label{cycle-IT-lemma}
Let $\ell = 3r + 1$ with $r \ge 1$. There exists a partition $\mathcal{P} = \{V_1, \ldots, V_{r+1} \}$ of the vertices of the $\ell$-cycle $C_{\ell}$ such that $|V_i| = 3$ for $1 \le i \le r - 1$, $|V_r| = |V_{r+1}| = 2$, and $C_{\ell}$ has no IT.
\end{lem}

\begin{proof}[Proof of Theorem \ref{cycles-no-IT}]
Let $\mathcal{P}_1$, $\mathcal{P}_2$, $\mathcal{P}_3$ be partitions of the vertices of $C_{\ell_1}, C_{\ell_2}, C_{\ell_3}$ given by Lemma \ref{cycle-IT-lemma}. Let $\{X_1, X_2\} \subseteq \mathcal{P}_1$, $\{Y_1, Y_2\} \subseteq \mathcal{P}_2$, $\{Z_1, Z_2\} \subseteq \mathcal{P}_3$ be the blocks of size $2$ in each partition. Now let $\mathcal{P}'$ be the partition of the vertices of $G' = C_{\ell_1} \sqcup C_{\ell_2}$ obtained by taking the union of $\mathcal{P}_1$ and $\mathcal{P}_2$ while distributing one vertex of $X_1$ into $Y_1$ and the other vertex of $X_1$ into $Y_2$. By Lemma \ref{join-lemma}, $G'$ has no IT with respect to $\mathcal{P}'$. Finally, let $\mathcal{P}$ be the partition of the vertices of $G = G' \sqcup C_{\ell_3}$ obtained by taking the union of $\mathcal{P}'$ and $\mathcal{P}_3$ while distributing one vertex of $X_2$ into $Z_1$ and the other vertex of $X_2$ into $Z_2$. By Lemma \ref{join-lemma}, $G$ has no IT with respect to $\mathcal{P}$, and we have $|V_i| = 3$ for every $V_i \in \mathcal{P}$.
\end{proof}

Aharoni, Holzman, Howard, and Spr\"ussel \cite{AhHoHoSp} proved Lemma \ref{cycle-IT-lemma} by describing the partition $\mathcal{P}$ explicitly. Here we give an alternative proof of Lemma \ref{cycle-IT-lemma} via an iterative approach similar to applications of Lemma \ref{join-lemma}. One could use this method to find a larger class of partitions $\mathcal{P}$ satisfying Lemma \ref{cycle-IT-lemma} than those described by the said authors. Our main tool is the following simple lemma.

\begin{lem} \label{edge-delete-lemma}
Let $G$ be a graph and let $\mathcal{P} = \{V_1, \ldots, V_r\}$ be a partition of $V(G)$. Suppose that $G$ has no IT with respect to $\mathcal{P}$. Let $uv \in E(G)$ be an edge with $u \in V_i$, $v \in V_j$, and $i \neq j$. Let $k \in \{1, \ldots, r\} \backslash \{i,j\}$, and let $F$ be a set of edges (possibly some in $G$) connecting every vertex in $V_k$ to either $u$ or $v$. Then the graph $H = G - uv \cup F$ has no IT with respect to $\mathcal{P}$.
\end{lem}

\begin{proof}
Assume for contradiction that $H$ has an IT $\{v_1, \ldots, v_r\}$ with respect to $\mathcal{P}$, where $v_\ell \in V_\ell$ for $1 \le \ell \le r$. Since $G$ has no IT, we must have that $v_i = u$ and $v_j = v$. But either $uv_k$ or $vv_k$ lies in $F$, which implies that $\{v_1, \ldots, v_r\}$ is not independent in $H$, a contradiction.
\end{proof}

\begin{proof}[Proof of Lemma \ref{cycle-IT-lemma}]
We use induction on $r$. For $r = 1$, we take $\mathcal{P}$ to be the standard bipartition of the $4$-cycle $C_4$, which satisfies the requirements. Let $r \ge 2$, and let $\mathcal{P}' = \{V_1, \ldots, V_r\}$ be a partition for the cycle $C_{3r - 2}$ given by the induction hypothesis. Let $\mathcal{Q} = \{ U = \{ u_1 \}, V_{r+1} = \{v_1,v_2\} \}$ be the standard bipartition of the path $K_{1,2}$. Let $G = C_{3r - 2} \sqcup K_{1,2}$, and let $\mathcal{P} = \{V_1, \ldots, V_{r - 2}, V_{r-1} \cup \{u_1\}, V_r, V_{r+1} \}$. By Lemma \ref{join-lemma}, $G$ has no IT with respect to $\mathcal{P}$. Let $e = uv$ be any edge of $C_{3r - 2}$, and let $F = \{ uv_1, vv_2 \}$. Then $G - uv \cup F$ is the $(3r + 1)$-cycle $C_{3r+1}$, and by Lemma \ref{edge-delete-lemma} $C_{3r+1}$ has no IT with respect to $\mathcal{P}$. See Figure \ref{seven-cycle} for illustration when $r = 2$. The blocks in $\mathcal{P}$ have the required sizes, so this finishes the proof.
\end{proof}

\begin{figure}
	\centering
	\includegraphics[scale=1]{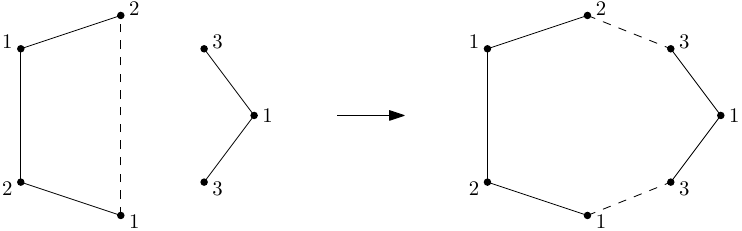}
	\caption{The disjoint union of a $4$-cycle $C_4$ and a $K_{1,2}$ (with no IT) being transformed into a $7$-cycle $C_7$ (with no IT). The label $i$ next to each vertex indicates which block $V_i$ the vertex lies in.}
	\label{seven-cycle}
\end{figure}

\section{Locally sparse graphs and list coloring} \label{section-locally-sparse-graphs}

In this section, we prove Theorem \ref{locally-sparse-graph} on locally sparse graphs, and then we construct new counterexamples to Reed's list coloring conjecture \cite{Re}, which was originally disproved by Bohman and Holzman \cite{BoHo}.

\subsection{Locally sparse graphs} \label{section-locally-sparse-graphs-1}

Let $G$ be a graph and let $\mathcal{P} = \{V_1, \ldots, V_r\}$ be a partition of $V(G)$. Recall that the \textit{local degree} of $G$ is the maximum, over all $v \in V(G)$ and $V_i \in \mathcal{P}$ not containing $v$, of the number of edges between $v$ and $V_i$. Also, the \textit{multiplicity} of $G$ is the maximum, over all connected components $C$ of $G$ and $V_i \in \mathcal{P}$, of the number of vertices in $C \cap V_i$. Loh and Sudakov \cite{LoSu} proved that for a graph $G$ and partition $\mathcal{P} = \{V_1, \ldots, V_r\}$ of $V(G)$, if $G$ has maximum degree $d$, local degree $m = o(d)$, and $|V_i| = d + o(d)$ for every $i$, then $G$ has an IT. We now prove Theorem \ref{locally-sparse-graph}, which complements this asymptotic result by giving a lower bound on the optimal error term $o(d)$ in $|V_i|$ depending on the local degree $m$.

Before proving Theorem \ref{locally-sparse-graph}, we note one obvious construction which will be used in the proof. Fix $m \ge 1$ and $r \ge 2$, and let $K_r(m)$ denote the complete $r$-partite graph with $m$ vertices in each part. Let $G_0$ be the disjoint union of $r - 1$ copies of $K_r(m)$, and define the partition $\mathcal{P}_0 = \{V_1, \ldots, V_r\}$ to be a standard $r$-partition of $G_0$, consisting of $r$ independent sets of size $(r-1)m$. Letting $d = (r - 1)m$ denote the maximum degree of $G_0$, we have that $|V_i| = (r - 1)m = d$ for every $i$, the local degree and multiplicity of $G_0$ are both $m$, and there is no IT with respect to $\mathcal{P}_0$.  

\begin{proof}[Proof of Theorem \ref{locally-sparse-graph}]
Assume that $d = (r - 1)m$, where $r \ge 2$. We start with any graph $H$ and partition $\mathcal{Q}$ coming from Theorem \ref{Szabo-Tardos}, say the Szab\'o-Tardos construction, in which $H$ is the disjoint union of $2d-1$ copies of $K_{d,d}$. Then $|U| = 2d - 1$ for every $U \in \mathcal{Q}$ and there is no IT. This graph has local degree and multiplicity $d$, which may be larger than $m$. We fix this by applying Lemma \ref{join-lemma} again. Fix a block $U \in \mathcal{Q}$. We add a copy of the vertex-partitioned graph $G_0, \mathcal{P}_0$ above, and then distribute vertices of $U$ into the $r$ blocks of $\mathcal{P}_0$. We do this equitably so that each of the $r$ blocks of $\mathcal{P}_0$ gets enlarged by at least  
\begin{align*}
	\left\lfloor \frac{2d - 1}{r} \right\rfloor = \left\lfloor \frac{2d-1}{d/m + 1} \right\rfloor =  \left\lfloor \frac{2m(d+m) - 2m^2 - m}{d+m} \right\rfloor = 2m - \left\lceil \frac{2m^2+m}{d+m} \right\rceil.
\end{align*}
Moreover, we do this in such a way that at most $m$ vertices of any connected component of $H$ get distributed to any block of $\mathcal{P}_0$. This is possible by a greedy distribution process, since at most $d$ vertices of $U$ lie in any connected component of $H$, and $\left\lceil \frac{d}{r} \right\rceil = \left\lceil \frac{(r - 1)m}{r} \right\rceil \le m$. After doing this procedure iteratively for every $U \in \mathcal{Q}$, we get a graph $G$ and partition $\mathcal{P}$ with local degree and multiplicity $m$, $|V_i| \ge d + 2m - \left\lceil \frac{2m^2+m}{d+m} \right\rceil$ for every $V_i \in \mathcal{P}$, and no IT.
\end{proof}

Note that the number of blocks in the graph constructed in the previous proof is $b(H) \cdot r$, where $b(H)$ denotes the number of blocks of the starting graph $H$. Slightly smaller constructions for Theorem \ref{locally-sparse-graph} are possible. Instead of the Szab\'o-Tardos construction, we could have started with any vertex-partitioned graph $H$ with no IT where every block has size at least $\left\lfloor \frac{2d-1}{r} \right\rfloor r = \left( 2m - \left\lceil \frac{2m^2+m}{d+m} \right\rceil \right) \left( \frac{d}{m} + 1 \right)$, and where every component has at most $d+1$ vertices from any block. Following the argument in Section \ref{section-disjoint-complete-bipartite}, one could form a smaller such graph $H$ by taking the disjoint union of possibly fewer than $2d - 1$ copies of $K_{d,d}$, and also possibly deleting some vertices. For example, when $m = 1$ we only need blocks of size at least $d+1$, and we can achieve these block sizes on the graph $H = K_{2,2} \sqcup K_{d,d} \sqcup K_{d,d}$ with $4$ blocks. Following the proof of Theorem \ref{locally-sparse-graph} on this latter graph $H$ yields a locally sparse graph constructed in \cite{AhHoHoSp}.

\begin{prob}
What is the optimal bound $M = M(d, m)$ such that whenever $G$ is a graph with maximum degree $d$ and $\mathcal{P} = \{V_1, \ldots, V_r\}$ is a partition of $V(G)$ with local degree $m$ and $|V_i| \ge M$ for every $i$, there exists an IT?
\end{prob}

\subsection{List coloring} \label{section-list-coloring}

Now we discuss list colorings and relate them to IT's in locally sparse graphs. Given a graph $H$ and a list assignment $L = (L(x) : x \in V(H))$ for $V(H)$, where each $L(x)$ is a set of colors, an \textit{$L$-coloring} is a map $\phi$ on $V(H)$ such that $\phi(x) \in L(x)$ for every vertex $x \in V(H)$. The $L$-coloring $\phi$ is \textit{proper} if $\phi^{-1}(c)$ is an independent set of $H$ for every color $c \in \bigcup_{x \in V(H)} L(x)$. As observed in \cite{Ha2}, the problem of finding a proper $L$-coloring of $H$ can be turned into a problem of finding an IT in an auxiliary graph $G$ with respect to some vertex partition $\mathcal{P}$. Specifically, $G$ has vertex set $\{(x, c) : x \in V(H), c \in L(x) \}$, and the pairs $(x, c)$ and $(y, c')$ form an edge in $G$ whenever $xy \in E(H)$ and $c = c'$. The partition $\mathcal{P}$ of $V(G)$ is $\{V_x : x \in V(H) \}$ where $V_x = \{(x, c) : c \in L(x)\}$. Then every proper $L$-coloring $\phi$ of $H$ corresponds to a unique IT $\{(x, \phi(x)) : x \in V(H)\}$ of $G$ with respect to $\mathcal{P}$. We say that the graph $G$ with vertex partition $\mathcal{P}$ is the \textit{list cover graph} arising from graph $H$ and list assignment $L$. 

The following characterization of list cover graphs is implicit in Bohman and Holzman \cite{BoHo}. We give a short proof for completeness.

\begin{lem}[\cite{BoHo}] \label{list-cover-graph}
A graph $G$ with partition $\mathcal{P}$ of $V(G)$ is a list cover graph if and only if
\begin{itemize}
	\item[(a)] for every connected component $C$ of $G$, every block of $\mathcal{P}$ contains at most one vertex of $C$ (that is, $G$ has multiplicity $1$); and
	\item[(b)] for every two connected components $C, C'$ of $G$, two blocks $V_i, V_j$ of $\mathcal{P}$, and vertices $v \in C \cap V_i$, $w \in C \cap V_j$, $v' \in C' \cap V_i$, $w' \in C' \cap V_j$, $v$ and $w$ are adjacent in $C$ if and only if $v'$ and $w'$ are adjacent in $C'$.
\end{itemize}
\end{lem}

\begin{proof}
Suppose that $G$ with vertex partition $\mathcal{P}$ is a list-cover graph arising from graph $H$ and list assignment $L$. Then every connected component $C$ of $G$ has vertex set $\{(x, c) : x \in U\}$ for some color $c$ and vertex subset $U \subseteq V(H)$. Then for $V_x = \{(x, c) : c \in L(x)\} \in \mathcal{P}$, we have $V_x \cap C = \{(x,c)\}$ if $x \in L(v)$, and $V_x \cap C = \emptyset$ otherwise, so that condition (a) holds. Moreover, for every two connected components $C, C'$ of $G$, two blocks $V_x, V_y$ of $\mathcal{P}$, and vertices $(x,c) \in C \cap V_x, (y,c) \in C \cap V_y, (x,c') \in C' \cap V_x, (y,c') \in C' \cap V_y$, we have that $(x,c)$ and $(y,c)$ are adjacent in $C$ if and only if $xy \in E(H)$, which is the case if and only if $(x,c')$ and $(y,c')$ are adjacent in $C'$. Thus condition (b) also holds.

Suppose conversely that (a) and (b) hold. Let $H$ be the graph with vertex set $\mathcal{P}$, where $V_i, V_j \in \mathcal{P}$ form an edge if and only if there exist vertices $v \in V_i, w \in V_j$
adjacent in $G$. Associate every connected component $C$ of $G$ with a distinct color $c$, and let $L$ be the list assignment for $V(H)$ where $L(V_i)$ is the set of colors $c$ associated with components $C$ where $C \cap V_i \neq \emptyset$. Then one can easily verify that $G, \mathcal{P}$ is the list-cover graph arising from $H, L$.
\end{proof}

Given a vertex $x \in V(H)$ and color $c \in L(x)$, the pair $(x, c)$ has \textit{color degree} equal to $|\{(y, c) : xy \in E(H), c \in L(y)\}|$. The \textit{maximum color degree} of $H, L$ is the maximum of the color degrees over all vertex-color pairs $(x, c)$. In other words, it is the maximum degree $d$ of the list cover graph $G$. Thus by Theorem \ref{max-degree-IT}, if $H$ has maximum color degree $d$ and $|L(x)| \ge 2d$ for every vertex $x \in V(H)$, then $H$ has a proper $L$-coloring (a result noted in \cite{Ha2}). Actually, because every list cover graph has local degree one, Loh and Sudakov's asymptotic theorem \cite{LoSu} implies that $H$ has a proper $L$-coloring whenever $|L(x)| \ge d + o(d)$ for every $x \in V(H)$ (a result proven before in \cite{ReSu}). Recall that Reed \cite{Re} conjectured that a proper $L$-coloring exists whenever $|L(x)| \ge d + 1$ for every $x \in V(H)$. This lower bound of $d+1$ is necessary, since for example there is no proper $L$-coloring if $H$ is the complete graph $K_{d+1}$ and $L$ is the uniform list assignment $L(x) = \{1, \ldots, d\}$ for every $x \in V(H)$. Bohman and Holzman \cite{BoHo} surprisingly found an example of a graph $H$ and list assignment $L$ for $V(H)$ such that $H, L$ has maximum color degree $d$, $|L(x)| = d+1$ for every $x \in V(H)$, and $H$ has no proper $L$-coloring. This disproved Reed's conjecture. More directly, they proved the following.

\begin{thm} \label{Bohman-Holzman}
For every $d \ge 2$, there exists a graph $G$ and a partition $\mathcal{P} = \{V_1, \ldots, V_r\}$ of $V(G)$ such that $G$ has maximum degree $d$, $|V_i| \ge d+1$ for every $i$, conditions (a) and (b) of Lemma \ref{list-cover-graph} hold, and there is no IT.
\end{thm}

The construction of Bohman and Holzman \cite{BoHo} for Theorem \ref{Bohman-Holzman} involved $2(d+1)(d^2+1)$ blocks. We describe a perhaps more intuitive proof of Theorem \ref{Bohman-Holzman} with only
$4(d+1)^2$ blocks.  

\begin{proof}

Fix $d \ge 2$. We start with any construction for Theorem \ref{locally-sparse-graph} for the case $m = 1$, so that each block has size at least $d+1$. For convenience, we delete some vertices to make all the blocks have size exactly $d+1$. Call the constructed graph $H$ and partition $\mathcal{Q}$. We have that $|U| = d+1$ for every $U \in \mathcal{Q}$, and we satisfy condition (a) of Lemma \ref{list-cover-graph} with $H$ and $\mathcal{Q}$, but not condition (b). To fix this, we again apply Lemma \ref{join-lemma}, adding additional collections of complete graphs. Let $G_0$ be the disjoint union of $d$ copies of the complete graph $K_{d+1}$, and let $\mathcal{P}_0$ be a standard $(d+1)$-partition of $G_0$. For every block $U \in \mathcal{Q}$ in succession, we add a copy of $G_0$ while distributing the $d+1$ vertices of $U$ into the $d+1$ blocks of $\mathcal{P}_0$, with each block of $\mathcal{P}_0$ receiving exactly one vertex. Doing this for each $U \in \mathcal{Q}$, we derive a graph $G$ and vertex partition $\mathcal{P}$ with no IT, and every block of $\mathcal{P}$ has size $d+1$. 

Observe that $G$ and $\mathcal{P}$ still satisfy condition (a) of Lemma \ref{list-cover-graph}. We claim that $G$ and $\mathcal{P}$ also satisfy condition (b). To see this, let $C$ and $C'$ be distinct components of $G$. If $C$ and $C'$ both lie in an added copy of $G_0$, then $C$ and $C'$ are cliques whose intersecting blocks are either the same or disjoint, so condition (b) holds for $C$ and $C'$. Next, suppose that the component $C$ lies in $H$ and component $C'$ lies in some added copy of $G_0$. Let $U$ be the block of $\mathcal{Q}$ whose vertices were distributed into the blocks intersecting $C'$ when we added a copy of $G_0$. Since $H$ has multiplicity $1$, there is at most one vertex in $U \cap C$, and thus there is at most one block of $\mathcal{P}$ intersecting both $C$ and $C'$. Hence condition (b) holds in this case as well. Finally, suppose that $C$ and $C'$ both lie in $H$. In this case, there are no blocks of $\mathcal{P}$ that intersect both $C$ and $C'$: If there were such a block, then it would have been constructed by taking a block $U$ of $\mathcal{Q}$ intersecting both $C$ and $C'$, and distributing its vertices into the blocks of $\mathcal{P}_0$ in such a way that the vertex in $U \cap C$ and the vertex in $U \cap C'$ end up in the same block of $\mathcal{P}$. But this is impossible because every block of $\mathcal{P}_0$ receives a unique vertex of $U$ by construction. Therefore, condition (b) holds for all $C$ and $C'$.
\end{proof}

The construction we described is similar in structure to that of Bohman and Holzman \cite{BoHo}. If we use the Szab\'o-Tardos construction as the starting point in the proof of Theorem \ref{locally-sparse-graph}, then the output of our proof of Theorem \ref{locally-sparse-graph} has $2d(d+1)$ blocks, and the output of our proof of Theorem \ref{Bohman-Holzman} has $2d(d+1)^2$ blocks. The $4(d+1)^2$ blocks claimed above can be achieved by following the proof of Theorem~\ref{locally-sparse-graph}, but using the graph $K_{2,2} \sqcup K_{d,d} \sqcup K_{d,d}$ with $4$ blocks of size $d+1$ as the starting point instead of the Szab\'o-Tardos construction. Our efforts were unsuccessful in achieving constructions where all the blocks have size at least $d+2$, a problem posed in \cite{BoHo, ReSu}.

\begin{prob}
Does there exist a universal constant $k \ge 1$ such that whenever $H, L$ has maximum color degree $d$ and $|L(x)| \ge d+k$ for every $x \in V(H)$, there exists a proper $L$-coloring of $H$? In particular, can we take $k = 2$?
\end{prob}

We comment on what occurs when we relax condition (a) or (b) in Lemma \ref{list-cover-graph}. In the case where we are only required to satisfy condition (a), this is the setting for $m = 1$ in Theorem \ref{locally-sparse-graph}. Even in this relaxed setting, we are unable to surpass the known lower bound $d+2$ (by Bohman and Holzman \cite{BoHo}) on the block sizes required to guarantee an IT in a graph with multiplicity one (a problem that was posed in \cite{AhHoHoSp}). On the other hand, if we are only required to satisfy condition (b), then the situation is quite different. For example, in \cite{CaHaKaWd} it is shown that there exist graphs $G$ and partitions $\mathcal{P}$ of $V(G)$ such that condition (b) holds, $|V_i| = 2d - 1$ for every $V_i \in \mathcal{P}$, and there is no IT. Thus, much of the difference between list cover graphs and general graphs lies in condition (a). However, if we require that condition (b) hold and we relax the local degree or multiplicity to being a fixed number $m \ge 1$, then the best we can currently achieve in general are block sizes $|V_i| \ge d + m$ whenever $d$ is a multiple of $m$ and $d > m$, obtained for example by ``blowing up" any construction for Theorem \ref{Bohman-Holzman}, replacing every vertex by $m$ independent vertices and replacing every edge by the complete bipartite graph $K_{m,m}$.

Finally, we remark that there are variations of list coloring for which independent transversals are even more relevant. Ordinary list colorings correspond to IT's in quite specific types of cover graphs (as described in Lemma \ref{list-cover-graph}), but a variation of list coloring known as \textit{conflict list coloring} corresponds to IT's in general graphs. For the definition and a broad overview of conflict (list) coloring, see \cite{DvEsKaOz1, DvEsKaOz2, FrHeKo}. One can phrase many famous variations of list coloring into the language of conflict list colorings and thus into IT's of certain classes of vertex-partitioned graphs. These variations include \textit{correspondence colorings} (also called \textit{DP-colorings}) \cite{DvPo}, \textit{adaptable list colorings} \cite{KoZh}, and \textit{cooperative list colorings} \cite{AhHoHoSp}. Each of these list coloring variations give rise to locally sparse cover graphs, and Theorem \ref{locally-sparse-graph} is relevant for each of them.

\section{Counterexamples} \label{section-questions}

In this section, we give counterexamples to two statements on IT's in graphs. One is a conjecture of Aharoni, Alon, and Berger \cite{AhAlBe} on IT's in $K_{1,k}$-free graphs. The other is a question of Aharoni, Holzman, Howard, and Spr\"ussel \cite{AhAlBe} on the structure of extremal vertex-partitioned graphs with no IT.

\subsection{A conjecture of Aharoni, Alon, and Berger}
Aharoni, Alon, and Berger \cite{AhAlBe} conjectured that for every $K_{1,k}$-free graph $G$ with maximum degree $d$ and partition $\mathcal{P} = \{V_1, \ldots, V_r\}$ of $V(G)$, if $|V_i| \ge d + k - 1$ for every $i$ then there exists an IT. We give a construction disproving their conjecture for $k \ge 3$. We start with the Szab\'o-Tardos construction $G_0, \mathcal{P}$ for Theorem \ref{Szabo-Tardos}, consisting of $2m - 1$ copies of $K_{m,m}$ where $m$ is a multiple of $k - 1$. To both parts of each complete bipartite component of $G_0$, we add $k - 1$ pairwise disjoint clique edge sets of size $\frac{m}{k - 1}$ on the vertex sets. The resulting graph $G$ is $K_{1,k}$-free with maximum degree $d = m + \frac{m}{k-1} - 1 = \frac{km}{k - 1} - 1$, each block of $\mathcal{P}$ has size $|V_i| = 2m - 1 = \frac{d(2k - 2) + k - 2}{k}$, and $G$ has no IT with respect to $\mathcal{P}$. For $k \ge 3$ and $d > k + \frac{2}{k - 2}$, these block sizes $|V_i|$ are greater than $d + k - 1$, which disproves the conjecture. 

On the other hand, Aharoni, Alon, and Berger \cite{AhAlBe} did prove that for a $K_{1,k}$-free graph with maximum degree $d$, blocks of size $|V_i| \ge \frac{d(2k - 3) + k - 1}{k - 1}$ for every $i$ guarantee the existence of an IT. We believe that the lower bound from our construction above is closer to being the optimal bound for the existence of IT's. We remark that the original conjecture of Aharoni, Alon, and Berger (Conjecture 2.8 in \cite{AhAlBe}) was actually about the topological connectedness of the independence complexes of $K_{1,k}$-free graphs, but this original conjecture implies the above conjecture about IT's. We can disprove their connectedness conjecture more directly. Let $m$ be a multiple of $k - 1$, and let $G$ be the complete bipartite graph $K_{m,m}$ together with $k - 1$ cliques of size $\frac{m}{k - 1}$ added to both parts of $K_{m,m}$. (This example was described in \cite{AhAlBe} but for a different problem.) Then the independence complex of $G$ is disconnected, but the number of vertices of $G$ is $2m = \frac{d(2k - 2) + 2k - 2}{k}$, which for $k \ge 3$ and $d > k - 1$ is greater than the at most $d + k - 1$ vertices predicted by the connectedness conjecture.

\begin{prob}
What is the optimal bound $N = N(k, d)$ such that whenever $G$ is a $K_{1,k}$-free graph with maximum degree $d$, and $\mathcal{P} = \{V_1, \ldots, V_r\}$ is a partition of $V(G)$ with $|V_i| \ge N$ for every $i$, there exists an IT?
\end{prob}

\subsection{A question of Aharoni, Holzman, Howard, and Spr\"ussel}
Aharoni, Holzman, Howard, and Spr\"ussel \cite{AhHoHoSp} asked whether, for every graph $G$ with maximum degree $d$ and partition $\mathcal{P} = \{V_1, \ldots, V_r\}$ of $V(G)$ such that $|V_i| \ge 2d - 1$ for every $i$ and there is no IT, there exists a $K_{d,d}$ component contained in the union of two blocks $V_i$, $V_j$. In this section, we give a construction that demonstrates a negative answer to this question. On the other hand, Theorem \ref{union-of-two-classes} in Section \ref{section-2blocks} implies that their question does hold when the graph $G$ is exactly the disjoint union of $2d - 1$ copies of $K_{d,d}$, as in Theorem \ref{Szabo-Tardos} (Corollary \ref{union-of-two-classes-Kdd}). In addition, Theorem \ref{union-of-two-classes} will also let us show that all possible constructions $G, \mathcal{P}$ for Theorem \ref{Szabo-Tardos} can be derived using the construction method described in this paper (Corollary \ref{characterize-Kdd}).

Using Lemma \ref{join-lemma}, it is straightforward to produce counterexamples to the question of Aharoni, Holzman, Howard, and Spr\"ussel \cite{AhHoHoSp}. For example, again start with the Szab\'o-Tardos \cite{SzTa} construction $G_0, \mathcal{P}_0$ for Theorem \ref{Szabo-Tardos}, which has $2d - 1$ copies of $K_{d,d}$. Fix any block $U \in \mathcal{P}_0$, and note that there is exactly one component $K$ of $G_0$ that intersects $U$ at $d$ vertices. We add a copy of the complete bipartite graph $K_{d,d-1}$ with the standard bipartition $\{A, B\}$, where $|A| = d$ and $|B| = d - 1$. We distribute $d - 1$ vertices of $U \cap K$ into $A$, and distribute the remaining $d$ vertices of $U$ into $B$. We do this for each of the $2d$ blocks of $\mathcal{P}_0$ successively. Let $G, \mathcal{P}$ be the final vertex-partitioned graph. Then $G$ has no IT with respect to $\mathcal{P}$ by Lemma~\ref{join-lemma}, every block $V_i \in \mathcal{P}$ has size $|V_i| = 2d - 1$, and there is no $K_{d,d}$ component of $G$ contained in the union of two blocks: every $K_{d,d}$ component in $G$ lies in $G_0$, and each of them is now separated into at least four blocks.

In the described counterexample, although there is no $K_{d,d}$ component contained in the union of two blocks, there are many $K_{d,d-1}$ components contained in the union of two blocks. Indeed, it is a feature of our construction method that whenever we apply Lemma \ref{join-lemma} by adding a complete bipartite component with the standard bipartition, this added complete bipartite component will still be contained in the union of two blocks. We could artificially break this property in our example by adding a new block $U$ of size $2d$, and attaching each vertex of $U$ to a unique $K_{d,d-1}$ component to make it a $K_{d,d}$ component. In this new example, every block has size at least $2d - 1$, there is no IT, and no component of the graph is contained in the union of two blocks. Of course, by adding the block $U$ we violate the block-minimality condition (a) required to apply Theorem \ref{union-of-two-classes}. See Section \ref{section-2blocks}. The following problem is also still open.

\begin{prob}
Let $G$ be a graph with maximum degree $d$ sufficiently large, and let $\mathcal{P} = \{V_1, \ldots, V_r\}$ be a partition of $V(G)$. If $|V_i| \ge 2d - 1$ for every $i$ and there is no IT, does $G$ necessarily contain a large complete bipartite subgraph contained in the union of two blocks (large meaning the parts have size $\Omega(d)$)?
\end{prob}

\section{Further constructions} \label{section-further}

In this section, we describe various further applications of our construction method. We begin with a continuation of the topic of the previous section.

\subsection{Components contained in the union of two blocks}\label{section-2blocks}

Let $G$ be a graph and let $\mathcal{P} = \{V_1, \ldots, V_r\}$ be a partition of $V(G)$. We say that $G$ is \textit{block-minimal with no IT} if $G$ has no IT with respect to $\mathcal{P}$ but $G - V_i$ has an IT with respect to $\mathcal{P} \setminus \{V_i\}$, for every $i \in [r]$. The following variation of the question of Aharoni, Holzman, Howard, and Spr\"ussel \cite{AhHoHoSp} does hold.

\begin{thm} \label{union-of-two-classes}
Let $G$ be a graph and let $\mathcal{P} = \{V_1, \ldots, V_r\}$ be a partition of $V(G)$, where $r \ge 2$. Suppose that
\begin{itemize}
	\item[(a)] $G$ is block-minimal with no IT,
	\item[(b)] $G$ is the disjoint union of complete bipartite graphs, and
	\item[(c)] the number of components of $G$ is $r - 1$.
\end{itemize}
Then there exist a component $K_{A,B}$ of $G$ and blocks $V_i, V_j$ in $\mathcal{P}$ such that $A \subseteq V_i$ and $B \subseteq V_j$.
\end{thm}

The proof of Theorem \ref{union-of-two-classes} is slightly technical and deferred to Section \ref{big-proof}. Here, we observe that Theorem \ref{union-of-two-classes} implies a positive answer for the question of Aharoni, Holzman, Howard, and Spr\"ussel \cite{AhHoHoSp} when the graph $G$ is exactly the disjoint union of $2d - 1$ copies of $K_{d,d}$, as in Theorem \ref{Szabo-Tardos}.

\begin{cor} \label{union-of-two-classes-Kdd}
For every construction $G, \mathcal{P}$ for Theorem \ref{Szabo-Tardos}, there exists a component $K_{d,d}$ of $G$ that is contained in the union of two blocks of $\mathcal{P}$. 
\end{cor}

\begin{proof}
Conditions (b) and (c) of Theorem \ref{union-of-two-classes} are satisfied by definition with $r = 2d$. Condition (a) is satisfied by Theorem \ref{many-copies-Kdd} when $d \neq 2$, and is satisfied by Theorem \ref{cycles-length-1-mod-3} when $d = 2$.
\end{proof}

\subsection{Characterizing disjoint unions of complete bipartite graphs with no IT} \label{section-characterize-complete-bipartite}

In this subsection, we describe a further consequence of Theorem \ref{union-of-two-classes} of interest, namely Theorem \ref{characterize-complete-bipartite}, which tells us that all possible vertex-partitioned graphs $G, \mathcal{P}$ satisfying conditions (a), (b), (c) of Theorem \ref{union-of-two-classes} can be derived by iteratively applying Lemma \ref{join-lemma}. In particular, all possible constructions $G, \mathcal{P}$ for Theorem \ref{Szabo-Tardos} can be derived by iteratively applying Lemma \ref{join-lemma} (Corollary \ref{characterize-Kdd}).

Graphs that are block-minimal with no IT were used in \cite{BeHaSz,Ha3} to give combinatorial proofs of Theorem \ref{max-degree-IT} and to obtain information about dominating sets in such graphs. In the lemma below, we observe that Lemma \ref{join-lemma} preserves the minimality properties of the vertex-partitioned graphs $G$ and $H$ involved.

\begin{lem} \label{join-lemma-minimal}
Let graphs $G, H, G \cup H$ and vertex partitions $\mathcal{P}, \mathcal{Q}, \mathcal{R}$ be as in Lemma \ref{join-lemma} (not assuming they have no IT's). Then with respect to the appropriate partitions, 
\begin{itemize}
	\item[(i)] if $G$ and $H$ are both block-minimal with no IT, then so is $G \cup H$;
	\item[(ii)] if $G \cup H$ and $H$ are both block-minimal with no IT, then so is $G$; and
	\item[(iii)] if $G \cup H$ and $G$ are both block-minimal with no IT, then so is $H$.
\end{itemize}
\end{lem}

\begin{proof}
We only prove (iii), as the proofs of (i) and (ii) are similar. Recall that $\mathcal{P} = \{V_1, \ldots, V_r\}$, $\mathcal{Q} = \{W_1, \ldots, W_s\}$, and $\mathcal{R} = \{V_1', \ldots, V_r', W_1, \ldots, W_{s-1}\}$. Suppose for contradiction that $H$ has an IT $\{w_1, \ldots, w_s\}$ with respect to $\mathcal{Q}$, where $w_j \in W_j$ for every $1 \le j \le s$. Suppose that $w_s \in V_i'$, where $1 \le i \le r$. Since $G$ is block-minimal with no IT, there exists an IT $\{v_1, \ldots, v_{i-1}, v_{i+1}, \ldots, v_r\}$ of $G - V_i$ with respect to $\mathcal{P} \setminus \{V_i\}$. Then $\{v_1, \ldots, v_{i-1}, w_j, v_{i+1}, \ldots, v_r, w_1, \ldots, w_{s-1} \}$ is an IT of $G \cup H$ with respect to $\mathcal{R}$, a contradiction. To prove the block-minimality of $H$, first note that by the block-minimality of $G \cup H$, for every $1 \le j \le s-1$, $(G \cup H) - W_j$ has an IT $\{v_1, \ldots, v_r, w_1, \ldots, w_{j-1}, w_{j+1}, \ldots, w_{s-1}\}$ with respect to $\mathcal{R} \setminus \{W_j\}$. Since $G$ has no IT with respect to $\mathcal{P}$, we have that $v_i \in W_s$ for some $1 \le i \le r$. Then $\{w_1, \ldots, w_{j-1}, w_{j+1}, \ldots, w_{s-1}, v_i\}$ is an IT of $H - W_j$ with respect to $\mathcal{Q} \setminus \{W_j\}$, as required. To check $W_s$, note again by the block-minimality of $G \cup H$, $(G \cup H) - V_1'$ has an IT $\{v_2, \ldots, v_r, w_1, \ldots, w_{s - 1}\}$ with respect to $\mathcal{R} \setminus \{V_1'\}$. Then $\{w_1, \ldots, w_{s-1}\}$ is an IT of $H - W_s$ with respect to $\mathcal{Q} \setminus \{W_s\}$. Therefore, $H$ is block-minimal with no IT. 
\end{proof}

Since $K_{d,d}$ with its standard bipartition $\mathcal{P}$ is block-minimal with no IT, Lemma \ref{join-lemma-minimal} implies that all the constructions we presented for Theorem \ref{Szabo-Tardos} are block-minimal with no IT. (This fact also follows from Theorem \ref{many-copies-Kdd}.) All of the other constructions in the proofs in this paper are also block-minimal with no IT. Now we use Theorem \ref{union-of-two-classes} to prove the main result of this subsection.

\begin{thm} \label{characterize-complete-bipartite}
Let $G$ be a graph and let $\mathcal{P} = \{V_1, \ldots, V_r\}$ be a partition of $V(G)$, where $r \ge 2$. Suppose that
\begin{itemize}
	\item[(a)] $G$ is block-minimal with no IT,
	\item[(b)] $G$ is the disjoint union of complete bipartite graphs, and
	\item[(c)] the number of components of $G$ is $r - 1$.
\end{itemize}
Then $G, \mathcal{P}$ can be constructed via Lemma \ref{join-lemma}, by iteratively adding complete bipartite graphs with the standard bipartition.
\end{thm}

\begin{proof}
We use induction on $r$. The case $r = 2$ is just our assumed starting point: $G$ is a complete bipartite graph and $\mathcal{P}$ is its standard bipartition. Suppose that $r \ge 3$. By Theorem \ref{union-of-two-classes}, there exist a component $K_{A, B}$ of $G$ and blocks $V_{r-1}, V_r \in \mathcal{P}$ (after relabeling) such that $A \subseteq V_{r-1}$ and $B \subseteq V_r$. Let $G' = G - (A \cup B)$ and let $\mathcal{P}' = \{V_1, \ldots, V_{r - 2}, V_{r-1}'\}$, where $V_{r-1}' = (V_{r-1} \cup V_r) - (A \cup B)$. 

Note that $V_{r-1}' \neq \emptyset$, as otherwise $A = V_{r-1}$ and $B = V_r$, and because $K_{A,B}$ has no IT with respect to $\{V_{r-1}, V_r\}$, this contradicts the assumption that $G$ is block-minimal with no IT. Thus, $\mathcal{P}'$ is a partition of $V(G')$. Observe that $G = G' \cup K_{A,B}$, and that $\mathcal{P}$ is obtained from $\mathcal{P}'$ and $\{A, B\}$ by distributing every vertex of $V_{r-1}'$ into either $A$ or $B$. That is, $G$ can be constructed from $G'$ and $K_{A,B}$ by applying Lemma \ref{join-lemma}. Since $G$ and $K_{A,B}$ are both block-minimal with no IT, by Lemma \ref{join-lemma-minimal}(iii) $G'$ is also block-minimal with no IT. Moreover, $G'$ is the disjoint union of $r - 2$ complete bipartite graphs. By induction, it follows that $G'$ and $\mathcal{P}'$ can be constructed via Lemma \ref{join-lemma}, by iteratively adding complete bipartite graphs with the standard bipartition. We conclude that the same is true for $G$ and $\mathcal{P}$.
\end{proof}

We remark that the converse of Theorem \ref{characterize-complete-bipartite} also holds: If $G, \mathcal{P}$ is constructed via Lemma \ref{join-lemma} by iteratively adding complete bipartite graphs with the standard bipartition, then conditions (a), (b), (c) are satisfied. This is observed by applying Lemma \ref{join-lemma-minimal}(i) iteratively. 

We note that condition (c) in Theorem \ref{union-of-two-classes} and Theorem \ref{characterize-complete-bipartite} is necessary. Figure \ref{block-minimal} shows an example of a vertex-partitioned graph where conditions (a) and (b) are satisfied but not (c), and the graph does not satisfy the conclusion of either theorem. Conditions (a) and (b) are more easily seen to be necessary.

\begin{cor} \label{characterize-Kdd}
All possible constructions $G, \mathcal{P}$ for Theorem \ref{Szabo-Tardos} can be derived via Lemma \ref{join-lemma}, by iteratively adding $K_{d,d}$ with the standard bipartition.
\end{cor}

\begin{proof}
Conditions (b) and (c) of Theorem \ref{characterize-complete-bipartite} are satisfied by definition with $r = 2d$. Condition (a) is satisfied by Theorem \ref{many-copies-Kdd} when $d \neq 2$, and is satisfied by Theorem \ref{cycles-length-1-mod-3} when $d = 2$.
\end{proof}

\begin{figure}
	\centering
	\includegraphics[scale=1]{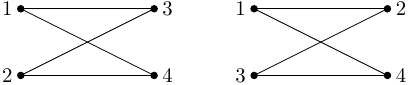}
	\caption{An example of a vertex-partitioned graph satisfying conditions (a) and (b) of Theorem \ref{union-of-two-classes} and Theorem \ref{characterize-complete-bipartite} but not satisfying (c) or the conclusion of either theorem.}
	\label{block-minimal}
\end{figure}

\subsection{Disjoint unions of general graphs} \label{section-general-graphs}
Here we state a general theorem about what happens when we successively apply Lemma \ref{join-lemma} to one graph $G$ and partition $\mathcal{P}$ with no IT. This result is motivated by the fact that many of our constructions are disjoint unions of the same graph, most notably $K_{d,d}$.

\begin{thm} \label{join-general-graphs-2}
Let $G$ be a graph and let $\mathcal{P} = \{V_1, \ldots, V_r\}$ be a partition of $V(G)$, such that $G$ has no IT with respect to $\mathcal{P}$. Suppose that the integer $n \ge 1$ satisfies
\begin{align*}
	\sum_{i \in I} |V_i| \ge n(|I| - 1) + 1 \hspace{5mm} \text{ for every } I \subseteq [r].
\end{align*}
Then there exist an integer $k \ge 1$ and a partition $\mathcal{R} = \{U_1, \ldots, U_{k(r-1)+1}\}$ of the vertices of the graph $G^k = G \sqcup \cdots \sqcup G$ (where $G$ appears $k$ times), such that $|U_i| \ge n$ for every $i$ and there is no IT. Specifically, if $J = \{i : |V_i| < n\} \neq \emptyset$ and $G_J = G \left[\bigcup_{i \in J} V_i \right]$, then we can take any 
\begin{align*}
	k \ge 1 + \sum_{i \in J} \left\lceil \frac{n - |V_i|}{|V(G_J)| - n(|J| - 1)} \right\rceil.
\end{align*}
\end{thm}

We sketch a proof of Theorem \ref{join-general-graphs-2}. We start with the graph $G$ and partition $\mathcal{P}$. Let $J = \{i : |V_i| < n\}$. If $J = \emptyset$, then $G$ and $\mathcal{P}$ satisfy the conclusion of the theorem with $k = 1$, and we are done. Otherwise, assume for convenience that $j \in J$ minimizes $|V_j|$. We will successively enlarge each block of $\mathcal{P}$ of size less than $n$ into a block of size at least $n$, by iteratively adding new copies of $G, \mathcal{P}$ and applying Lemma \ref{join-lemma}. At a given step, we fix any block $U$ of size less than $n$. We assume by induction that $|U| \ge |V_j|$. We proceed by adding a copy of $G, \mathcal{P}$ to the current graph and distributing the vertices of $U$ into the copy of the blocks $\{V_i : i \in J\}$ in this added copy of $\mathcal{P}$. We distribute $n - |V_i|$ vertices of $U$ to $V_i$ for each $i \in J \setminus \{j\}$ to create blocks of size $n$, and we distribute the remaining vertices of $U$ to $V_j$ to create one block $U'$ of size at least $|U| + 1$. It is possible to distribute vertices this way because by the inequality $|U| \ge |V_j|$ and our assumption on $n$, we have that $|U| - \sum_{i \in J \setminus \{j\}} (n - |V_i|) \ge 1$ and $|V_j| + \left( |U| - \sum_{i \in J \setminus \{j\}} (n - |V_i|) \right) \ge |U| + 1$. Thus, at each step we create at most one new block $U'$ of size less than $n$, and the size of $U'$ is at least one more than the size of the block $U$ from the previous step. Therefore, after adding sufficiently many copies of $G, \mathcal{P}$, all the blocks will have size at least $n$, and there will be no IT by Lemma \ref{join-lemma}. A careful analysis shows that the number of copies $k$ of $G$ involved is given by the lower bound in the theorem statement, and we may similarly add further copies of $G, \mathcal{P}$ to have the number of copies of $G$ be any $k$ as claimed.

\bigskip

Now we describe a few applications of Theorem \ref{join-general-graphs-2}. When $G  = K_{d,d}$ and $\mathcal{P}$ is the standard bipartition, we can take $n = 2d - 1$ and $k = 2d - 1$ in Theorem \ref{join-general-graphs-2}. This means that the disjoint union of $k = 2d - 1$ copies of $K_{d,d}$ has a partition into blocks of size $n = 2d - 1$ with no IT. Thus, Theorem \ref{join-general-graphs-2} generalizes Theorem \ref{Szabo-Tardos}. In fact, for any $1 \le n \le 2d - 1$, Theorem \ref{join-general-graphs-2} says that we can achieve blocks of size at least $n$ with no IT on the disjoint union of $k = 2\left\lceil \frac{d}{2d - n} \right\rceil - 1$ copies of $K_{d,d}$. We can restate this as follows.

\begin{cor} \label{general-Szabo-Tardos}
For integers $n \ge 1$ and $r \ge 2$ even, let $d =  \left\lceil \frac{rn}{2(r - 1)} \right\rceil$ and let $G$ be the disjoint union of $r - 1$ copies of $K_{d,d}$. There exists a partition $\mathcal{P} = \{V_1, \ldots, V_r\}$ of $V(G)$ such that $|V_i| \ge n$ for every $i$ and there is no IT.
\end{cor}


Corollary \ref{general-Szabo-Tardos} represents the more general version of Theorem \ref{Szabo-Tardos} proven by Szab\'o and Tardos \cite{SzTa}. It is the result they used to solve one direction of an extremal problem of Bollob\'as, Erd\H{o}s, and Szemer\'edi \cite{BoErSz} mentioned in the Introduction (the other direction being solved in \cite{HaSz1}). In terms of $d \ge 1$ and $r \ge 2$, the block size lower bound $n$ in Corollary \ref{general-Szabo-Tardos} is equal to $\left\lfloor \frac{2d(r - 1)}{r} \right\rfloor$.

For another application of Theorem \ref{join-general-graphs-2}, recall that $K_r(m)$ denotes the complete $r$-partite graph with $m$ vertices in each part. The disjoint union of $r - 1$ copies of $K_{r}(m)$ has no IT with respect to a standard $r$-partition. Applying Theorem \ref{join-general-graphs-2} to this graph gives the following.

\begin{cor} \label{complete-multipartite}
Let $r \ge 2, m \ge 1$. Let $G$ be the disjoint union of $(r - 1)\left(1 + r \left\lceil \frac{m - 1}{r - 1} \right\rceil \right)$ copies of $K_{r}(m)$, which has maximum degree $d = (r - 1)m$. There exists a partition $\mathcal{P} = \{V_1, \ldots, V_s\}$ of $V(G)$ such that $|V_i| \ge rm - 1 = d + m - 1$ for every $i$ and there is no IT.
\end{cor}

\subsection{Generalization to simplicial complexes} \label{section-general-complexes}
In this subsection, we note how to generalize Lemma \ref{join-lemma} and our construction method to all simplicial complexes, and outline a few applications.

A \textit{simplicial complex} (or a \textit{complex} for short) $\mathcal{C}$ on finite ground set $V = V(\mathcal{C})$ is a collection of subsets of $V$ such that if $A \in \mathcal{C}$ and $B \subseteq A$, then $B \in \mathcal{C}$. The sets in $\mathcal{C}$ are usually called \textit{simplices}. For example, the collection $\mathcal{I}(G)$ of independent sets of a graph $G$ is a complex, called the \textit{independence complex} of $G$. Given two complexes $\mathcal{C}$ and $\mathcal{D}$ on disjoint ground sets, their \textit{join} is the simplicial complex $\mathcal{C} \ast \mathcal{D} = \{A \cup B : A \in \mathcal{C} \text{ and } B \in \mathcal{D} \}$. For example, if $G, H$ are disjoint graphs and $\mathcal{I}(G), \mathcal{I}(H)$ are their independence complexes, then $\mathcal{I}(G \cup H) = \mathcal{I}(G) \ast \mathcal{I}(H)$.

Given a simplicial complex $\mathcal{C}$ and a partition $\mathcal{P} = \{V_1, \ldots, V_r\}$ of $V(\mathcal{C})$, we wish to decide whether $\mathcal{C}$ contains a transversal of $\mathcal{P}$, which we call a \textit{multicolored simplex}. This generalizes the concept of an independent transversal, which is the special case where $\mathcal{C}$ is the independence complex of a graph. We can easily generalize Lemma \ref{join-lemma} and Theorem \ref{join-general-graphs-2} as follows, using the same proofs.

\begin{lem} \label{join-lemma-complex}
Let $\mathcal{C}$ and $\mathcal{D}$ be complexes on disjoint ground sets, and let $\mathcal{P} = \{V_1, \ldots, V_r\}$ and $\mathcal{Q} = \{W_1, \ldots, W_s\}$ be partitions of $V(\mathcal{C})$ and $V(\mathcal{D})$ respectively. Let $\mathcal{R} = \{V_1', \ldots, V_r', W_1, \ldots, W_{s-1}\}$, where $V_1' \supseteq V_1, \ldots, V_r' \supseteq V_r$ are obtained by distributing each of the vertices of $W_s$ into one of $V_1, \ldots, V_r$ arbitrarily. If $\mathcal{C}$ has no multicolored simplex with respect to $\mathcal{P}$ and if $\mathcal{D}$ has no multicolored simplex with respect to $\mathcal{Q}$, then $\mathcal{C} \ast \mathcal{D}$ has no multicolored simplex with respect to $\mathcal{R}$.
\end{lem}

\begin{thm} \label{join-general-complexes-2}
Let $\mathcal{C}$ be a complex and let $\mathcal{P} = \{V_1, \ldots, V_r\}$ be a partition of $V(\mathcal{C})$, such that $\mathcal{C}$ has no multicolored simplex with respect to $\mathcal{P}$. Suppose that integer $n \ge 1$ satisfies
\begin{align*}
	\sum_{i \in I} |V_i| \ge n(|I| - 1) + 1 \hspace{5mm} \text{ for every } I \subseteq [r].
\end{align*}
Then there exist an integer $k \ge 1$ and a partition $\mathcal{R} = \{U_1, \ldots, U_{k(r-1)+1}\}$ of the ground set of the complex $\mathcal{C}^k = \mathcal{C} \ast \cdots \ast \mathcal{C}$ (where $\mathcal{C}$ appears $k$ times), such that $|U_i| \ge n$ for every $i$ and $\mathcal{C}^k$ has no multicolored simplex with respect to $\mathcal{R}$. Specifically, if $J = \{i : |V_i| < n\} \neq \emptyset$ and $\mathcal{C}_J = \mathcal{C} \left[\bigcup_{i \in J} V_i \right]$, then we can take any 
\begin{align*}
	k \ge 1 + \sum_{i \in J} \left\lceil \frac{n - |V_i|}{|V(\mathcal{C}_J)| - n(|J| - 1)} \right\rceil.
\end{align*}
\end{thm}

One simplicial complex which can be studied using these ideas is that of independent transversals (IT's) in uniform hypergraphs. Given a hypergraph $H$ and a partition $\mathcal{P}$ of $V(H)$, we wish to find a transversal of $\mathcal{P}$ that is an independent set of $H$ (does not contain an edge). Let $H_0=K_r^r(m)$ be the complete $r$-partite $r$-uniform hypergraph with $m$ vertices in each part, and let $\mathcal{P}_0$ be the standard $r$-partition of $V(H_0)$. Note that $H_0$ has no IT with respect to $\mathcal{P}_0$. Applying Theorem \ref{join-general-complexes-2} with $\mathcal{C}$ being the independence complex of $H_0$, we derive an example of a hypergraph $H$ and partition $\mathcal{P}$ such that $|V_i| \ge n = \left\lfloor \frac{rm - 1}{r - 1} \right\rfloor$ for every $V_i \in \mathcal{P}$ and there is no IT. Letting $d = m^{r - 1}$ be the maximum degree of $H$, we find that $d \ge \left( \frac{(r - 1)n + 1}{r} \right)^{r - 1}$. We leave the following as an open problem.

\begin{prob}
For an $r$-uniform hypergraph $G$ with maximum degree $d$ and partition $\mathcal{P} = \{V_1, \ldots, V_n\}$ of $V(G)$, do the conditions $|V_i| \ge n$ for every $i$ and $d \le \left( 1 - \frac{1}{r} \right)^{r - 1}n^{r-1}$ guarantee the existence of an IT?
\end{prob}

For context, a hypergraph generalization of a Lov\'asz local lemma argument by Alon \cite{Al2} shows that $d \le \frac{1}{er}n^{r - 1}$ implies the existence of an IT. This result is slightly improved by Wanless and Wood \cite{WaWo}, who showed that $d \le \frac{1}{r}(1 - \frac{1}{r})^{r-1} n^{r - 1}$ implies the existence of an IT. They proved this result more strongly when $d$ represents the \textit{maximum block average degree} of the vertex-partitioned graph, and moreover they showed that there exist exponentially many IT's in this case. Groenland, Kaiser, Treffers, and Wales \cite{GrKaTrWa} later showed that in terms of the maximum block average degree, the bound of Wanless and Wood is asymptotically best possible. For the case of graphs $r = 2$, it was shown in \cite{HaWd} that we can effectively interpolate between the optimal maximum degree and maximum block average degree bounds on the block sizes. But for hypergraphs, we still do not know the optimal bound in terms of just the maximum degree. 

A special case of the hypergraph IT problem is that of $J$-free transversals of graphs, where $J$ is a fixed $k$-regular graph. This problem was originally studied by Szab\'o and Tardos \cite{SzTa}. Given a vertex-partitioned graph $G$, we seek a transversal $T$ of the blocks such that the induced subgraph $G[T]$ does not contain a subgraph isomorphic to $J$. This time the parameter of interest is the maximum degree $d$ of the graph $G$. Independent transversals correspond to the case $J = K_2$, and the general case is the hypergraph IT problem in which $r=|V(J)|$ and the edges of the $r$-uniform hypergraph $H$ are those $r$-subsets of $V(G)$ that induce a subgraph containing an isomorphic copy of $J$. Here the $r$-uniform hypergraph $H_0=K_r^r(m)$ can be associated with the graph $J(m)$, the $m$-fold \textit{blow-up} of $J$, which is the graph obtained from $J$ by replacing every vertex of $J$ by an independent set of size $m$ and replacing every edge by a copy of the complete bipartite graph $K_{m,m}$. When $d$ is a multiple of $k$, the above argument applied with $J(d/k)$ gives a graph $G$ with maximum degree $d$ and partition $\mathcal{P}$ such that $|V_i| \ge \left\lceil \frac{|V(J)|}{(|V(J)| - 1)k } d \right\rceil - 1$ for every $V_i \in \mathcal{P}$ and there is no $J$-free transversal. This lower bound on the block sizes was also derived by Szab\'o and Tardos \cite{SzTa}, and the following problem is still open.

\begin{prob}
For a graph $G$ with maximum degree $d$, partition $\mathcal{P} = \{V_1, \ldots, V_r\}$ of $V(G)$, and regular graph $J$, does the condition $|V_i| \ge \left\lceil \frac{|V(J)|}{(|V(J)| - 1)k } d \right\rceil$ for every $i$ guarantee the existence of a $J$-free transversal?
\end{prob}

The case where $J$ is a clique may be of particular interest. Several further settings in which these arguments apply, including rainbow matchings, are explored in detail in \cite{WdThesis}.

\section{Proof of Theorem \ref{union-of-two-classes}} \label{big-proof}
In this section, we prove Theorem \ref{union-of-two-classes}. Our proof adapts combinatorial methods from \cite{HaSz1, HaSzTa}. We will follow the same argument and terminology as \cite{HaSz1} for the construction of feasible pairs, but the series of claims after Claim \ref{claim-1} are more specific to our setting.

Let graph $G$ and vertex partition $\mathcal{P} = \{V_1, \ldots, V_r\}$ satisfy conditions (a), (b), (c) of Theorem \ref{union-of-two-classes}. A \textit{partial} independent transversal is an independent set of $G$ containing at most one vertex from each block $V_i$. Let $\mathcal{T}$ be the set of all partial independent transversals of $G$. For a partial independent transversal $T \in \mathcal{T}$ and a vertex $w \in V(G) \setminus T$, let $C(w, T)$ denote the vertex set of the component of $G[\{w\} \cup T]$ containing $w$. Then $G[C(w, T)]$ is a star with center $w$. A block $V_i$ is \textit{active} for a subset $I \subseteq V(G)$ if $V_i \cap I \neq \emptyset$. Let $S(I)$ denote the set of active blocks for $I$. The \textit{block graph} $\mathcal{G}_I$ of $I$ is obtained from $G[I]$ by contracting, for each $V_i \in S(I)$, the vertices of $V_i \cap I$ into a single vertex. Thus the vertex set of $\mathcal{G}_I$ is $S(I)$. A \textit{non-trivial} star is a star with at least two vertices. We say that a pair $(I, T)$ is \textit{feasible} if
\begin{itemize}
	\item[(i)] $I \subseteq V(G)$ and $T \in \mathcal{T}$ is a partial independent transversal of maximum size;
	\item[(ii)] $G[I]$ is a forest whose components are the $|W|$ non-trivial vertex-disjoint stars $C(w, T)$, for $w \in W = I \setminus T$; and
	\item[(iii)] the block graph $\mathcal{G}_I$ is a tree on vertex set $S(I)$.
\end{itemize}
See Figure 2 in \cite{HaSz1} for an illustration of a feasible pair. Note that $(\emptyset, T)$ is a feasible pair for any $T \in \mathcal{T}$. Our goal is to construct a feasible pair $(I, T)$ that is an \textit{induced matching configuration} (IMC), meaning that $G[I]$ is a perfect matching and $\mathcal{G}_I$ is a tree on $r$ vertices (meaning that every block of $\mathcal{P}$ is active). See Figure 1 in \cite{HaSz1} for an illustration of an IMC (though unlike there, we will still distinguish between $T$ and $I \setminus T$). IMC's are the main structures that allow us to prove Theorem \ref{union-of-two-classes}.

To construct an IMC, we follow the same algorithm as in \cite{HaSz1}, Section 2. Start by fixing a block $V_i \in \mathcal{P}$. By condition (a) of Theorem \ref{union-of-two-classes}, there exists a maximum partial independent transversal $T_0 \in \mathcal{T}$ that intersects every block of $\mathcal{P}$ except $V_i$. Our initial pair is $(\emptyset, T_0)$. Suppose that we currently have the pair $(I, T)$. If $I = \emptyset$, we let $w$ be any vertex in $V_i$. Otherwise, suppose there exists a vertex $w \in \bigcup_{V_j \in S(I)} V_j$ (in one of the active blocks) that is not in $I$ and is not adjacent to any vertex in $I$. Let $T' \in \mathcal{T}$ be a maximum partial independent transversal that, subject to the conditions that $T'$ agrees with $T$ inside $S(I)$ and also satisfies $(T' \setminus I) \cap N_G(I) = \emptyset$, minimizes the degree $\text{deg}_{T'}(w)$ of $w$ into $T'$. Then we update the pair $(I, T)$ to $(I \cup C(w, T'), T')$. We iterate until there are no more choices of vertices $w$ as above. Note that $I$ grows by at least one vertex at each iteration, so the procedure terminates.

\begin{claim} \label{claim-1}
A feasible pair $(I, T)$ is maintained throughout the above procedure.
\end{claim}

See \cite{HaSz1} for a proof of Claim \ref{claim-1}. Let $(I, T)$ be the feasible pair once the procedure terminates. Now we prove a series of claims showing that $(I, T)$ is an IMC. Let $U(I) = \bigcup_{V_j \in S(I)} V_j$ be the set of vertices that lie in an active block for $I$.

\begin{claim} \label{claim-2}
Every vertex in $U(I)$ lies in a component of $G$ intersecting $I$.
\end{claim}

\begin{proof}
If $v \in U(I)$ and $v$ is in a component not intersecting $I$, then $v \notin I$ and $v$ is independent of $I$, so we could have continued our procedure with the vertex $v$, contradicting that $(I, T)$ is the feasible pair once the procedure terminated.
\end{proof}

\begin{claim} \label{claim-3}
Every vertex in $T \setminus I$ lies in a component of $G$ not intersecting $I$.
\end{claim}

\begin{proof}
Suppose that vertex $v \in T \setminus I$ is in a component of $G$ intersecting $I$. Because every component of $G$ is a complete bipartite graph (condition (b) of Theorem \ref{union-of-two-classes}), $v$ is adjacent to a vertex in either $I \cap T$ or $I \setminus T$ (using condition (ii) of feasible pairs), so in particular $v$ is adjacent to a vertex in $I$. Then $(T \setminus I) \cap N_G(I) \neq \emptyset$, which contradicts our assumption on how $T \in \mathcal{T}$ was chosen.
\end{proof}

\begin{claim} \label{claim-4}
$G[U(I)]$ has no IT with respect to $S(I)$.
\end{claim}

\begin{proof}
Suppose that $G[U(I)]$ has an IT $T'$ with respect to $S(I)$. Then $T' \cup (T \setminus I)$ is a transversal of $\mathcal{P}$, and it is independent in $G$ by Claim \ref{claim-2} and Claim \ref{claim-3}. Thus $T' \cup (T \setminus I)$ is an IT of $G$ with respect to $\mathcal{P}$, contradicting condition (a) of Theorem \ref{union-of-two-classes}.
\end{proof}

\begin{claim} \label{claim-5}
We have $S(I) = \mathcal{P}$, and thus $U(I) = V(G)$ and $T \subseteq I$.
\end{claim}

\begin{proof}
Because $G$ is block-minimal with no IT by condition (a), Claim \ref{claim-4} implies that $S(I) = \mathcal{P}$.
\end{proof}

\begin{claim} \label{claim-6}
Every vertex $w \in I \setminus T$ has degree $\text{deg}_T(w) = 1$ into $T$.
\end{claim}

\begin{proof}
By Claim \ref{claim-2} and Claim \ref{claim-5}, every component of $G$ intersects $I$, and by construction every such component contains exactly one vertex from $I \setminus T$. Therefore,
\begin{align*}
	r - 1 = |T| = \sum_{w \in I \setminus T} \text{deg}_T(w) \ge \sum_{w \in I \setminus T} 1 = |I \setminus T| = \text{\# components of } G = r - 1,
\end{align*}
where $\text{deg}_T(w) \ge 1$ by condition (ii) of feasible pairs, and the last equality is by condition (c) of Theorem \ref{union-of-two-classes}. Hence, $\text{deg}_T(w) = 1$ for every $w \in I \setminus T$.
\end{proof}

\begin{claim} \label{claim-7}
	$(I, T)$ is an IMC.
\end{claim}

\begin{proof}
By condition (ii) of feasible pairs and Claim \ref{claim-6}, every component of $G[I]$ is a $2$-vertex star, i.e., a single edge, meaning that $G[I]$ is a perfect matching. By Claim \ref{claim-5}, every block of $\mathcal{P}$ is active and thus $\mathcal{G}_I$ is a tree on vertex set $\mathcal{P}$.
\end{proof}

We have shown that we can construct an IMC $(I, T)$ with $V_i \cap T = \emptyset$ for any initial block $V_i \in \mathcal{P}$. Now we observe additional facts that we will utilize in our proof of Theorem \ref{union-of-two-classes}.

\begin{claim} \label{claim-8}
Every vertex of $G$ has exactly one neighbor in $I$.
\end{claim}

\begin{proof}
By Claim \ref{claim-2} and Claim \ref{claim-5}, every vertex $v \in V(G)$ lies in a component intersecting $I$. Since $G[I]$ is a perfect matching (by Claim \ref{claim-7}) and every component of $G$ is a complete bipartite graph (by condition (b) of Theorem \ref{union-of-two-classes}), it follows that $v$ is adjacent to exactly one vertex in $I$.
\end{proof}

For the remaining part of the proof, we need some additional terminology. Let $(I, T)$ be an IMC with $V_i \cap T = \emptyset$. We view the block graph $\mathcal{G}_I$ as a rooted tree whose root is $V_i$ with respect to $T$, and every other block is a descendant of $V_i$. For a block $V_j$ and vertex $w \in V_j \cap (I \setminus T)$, the \textit{$w$-child} of $V_j$ in $\mathcal{G}_I$ is the block $V_k$ containing the unique neighbor of $w$ in $T$. We say that a block $V_k$ is a \textit{$w$-descendant} of $V_j$ if there is a sequence of blocks $V_j = U_1, U_2, \ldots, U_{\ell - 1}, U_\ell = V_k$ together with a sequence of vertices $w = w_1, v_2, w_2, \ldots, v_{\ell - 1}, w_{\ell - 1}, v_\ell$ such that $w_n \in U_n \cap (I \setminus T)$, $v_n \in U_n \cap T$, and $v_n$ is the unique neighbor of $w_{n - 1}$ in $T$, for every $n$. In particular, each $U_n$ is the $w_{n-1}$-child of $U_{n-1}$. The sequence of blocks $U_1, \ldots, U_\ell$ is the unique path from $V_j$ to $V_k$ in the block graph $\mathcal{G}_I$.

When $V_j = V_i$ is the root of $\mathcal{G}_I$ with respect to $T$, the set $T' = T \cup \{w_1, \ldots, w_{\ell-1}\} - \{v_2, \ldots, v_{\ell}\}$ is another maximum partial independent transversal. This is an application of the fact that $G[I]$ is a perfect matching. Moreover, $(I, T')$ is an IMC with root $V_k$ (as $V_k \cap T' = \emptyset$). We say that $T'$ is obtained from $T$ by \textit{switching} along the path from $V_i$ to $V_k$ in $\mathcal{G}_I$.

\begin{claim} \label{claim-9}
For every vertex $w \in I \setminus T$, say $w \in V_j$, every neighbor of $w$ in $G$ lies in a block that is a $w$-descendant of $V_j$ in $\mathcal{G}_I$.
\end{claim}

\begin{proof}
Suppose for contradiction that $w$ has a neighbor $x$ in a block $V_k$ that is not a $w$-descendant of $V_j$ in $\mathcal{G}_I$. Let $T'$ be the maximum partial independent transversal obtained from $T$ by switching along the path from $V_i$ to $V_k$ in $\mathcal{G}_I$, so that $T' \cap V_k = \emptyset$. Because $V_k$ is not a $w$-descendant of $V_j$, we have that $w \in I \setminus T'$. Then $T' \cup \{x\}$ is independent in $G$ because $w \in I \setminus T'$ and $w$ is the unique neighbor of $x$ in $I$ (applying Claim \ref{claim-8}). Since $x \in V_k$, it follows that $T' \cup \{x\}$ is an IT of $G$, a contradiction. 
\end{proof}

We now prove the following lemma, from which we easily derive Theorem \ref{union-of-two-classes}.

\begin{lem} \label{complete-bipartite-lemma}
Let graph $G$ and partition $\mathcal{P} = \{V_1, \ldots, V_r\}$ satisfy conditions (a), (b), (c) of Theorem \ref{union-of-two-classes}. Then for every block $V_i$, there exists a component $K_{A, B}$ of $G$ such that $A \subseteq V_i$.
\end{lem}

\begin{proof}
Let $V_i$ be any block of $\mathcal{P}$. We wish to find a vertex $v \in V(G)$ such that $N_G(v) \subseteq V_i$. If we succeed, then because every component of $G$ is a complete bipartite graph (by condition (b)), it follows that $N_G(v)$ is a part $A$ of some complete bipartite component $K_{A, B}$ of $G$, and $A \subseteq V_i$ as required.
	
Applying conditions (a), (b), (c) as we have done, let $(I, T)$ be an IMC such that $V_i \cap T = \emptyset$, i.e., $V_i$ is a root of the block graph $\mathcal{G}_I$ with respect to $T$. First suppose that $V_i$ is a leaf of $\mathcal{G}_I$ (allowing the root to be considered a leaf). Let $V_j$ be the unique neighbor of $V_i$ in $\mathcal{G}_I$. Let $T' = T \cup \{w\} - \{v\}$ be obtained from $T$ by switching along the path from $V_i$ to $V_j$, where $w \in V_i \cap (I \setminus T)$ and $v \in V_j \cap T$. Now $\mathcal{G}_I$ has root $V_j$ with respect to $T'$. Then the block $V_i$ is the unique $v$-descendant of $V_j$ in $\mathcal{G}_I$, so by Claim \ref{claim-9} we have that $N_G(v) \subseteq V_i$, as we wanted.
	
Now suppose that $V_i$ is not a leaf of $\mathcal{G}_I$. Let $w \in V_i \cap (I \setminus T)$, let $V_j$ be the $w$-child of $V_i$, and let $v \in V_j \cap T$ be the neighbor of $w$ in $T$. If $N_G(v) \subseteq V_i$, then we are done. Otherwise, let $x \in N_G(v) \setminus V_i$. Note that $x \notin V_j$, as otherwise $T - \{v\} \cup \{w, x\}$ is an IT of $G$ with respect to $\mathcal{P}$ (using condition (b)), which contradicts condition (a) that there is no IT. Let $I' = I - \{w\} \cup \{x\}$. We claim that $(I', T)$ is another IMC. To see this, first observe that $v$ is the unique neighbor of $x$ in $I$ by Claim \ref{claim-8}, which implies that $G[I']$ is still a perfect matching. Also, let $T''$ be obtained from $T$ by switching along the single-edge path from $V_i$ to $V_j$ in $\mathcal{G}_{I}$. Then in the IMC $(I, T'')$, the vertex $x$ lies in a block that is a $v$-descendant of $V_j$, by Claim \ref{claim-9}. From this we can conclude that the new block graph $\mathcal{G}_{I'}$ is still a tree on $r$ vertices, rather than the disjoint union of a unicyclic component and a tree. We deduce that $(I', T)$ is still an IMC.

Now, $V_i$ has degree one less in $\mathcal{G}_{I'}$ than it had in $\mathcal{G}_I$. Therefore, applying the above procedure iteratively, we produce a sequence of IMC's where the degree of $V_i$ decreases by one at each step. Thus before or at the base case where $V_i$ is a leaf, we will find a desired vertex $v \in V(G)$ such that $N_{G}(v) \subseteq V_i$.
\end{proof}

\begin{proof}[Proof of Theorem \ref{union-of-two-classes}]
By Lemma \ref{complete-bipartite-lemma}, each of the $r$ blocks in $\mathcal{P}$ covers one side of one out of $r - 1$ complete bipartite components of $G$. Therefore, by the pigeonhole principle, there exist two distinct blocks $V_i$ and $V_j$ and a single component $K_{A,B}$ such that $A \subseteq V_i$ and $B \subseteq V_j$.
\end{proof}

\end{document}